\documentclass[12pt]{article}
\usepackage{amsmath}
\usepackage{amsfonts}
\usepackage{amsthm}
\usepackage{graphicx}
\graphicspath{{Figures/}} 
\usepackage{overpic}
\usepackage{amssymb} 
\usepackage{pictex}
\usepackage{rotating}  
\usepackage{color}
\def\blue{\color{blue}}
\def\red{\color{red}}

\usepackage{cite} 

\usepackage{enumitem} 

\usepackage{accents} 

\usepackage{yhmath} 
\usepackage{etex} 

\usepackage{comment}
\numberwithin{equation}{section}

\newtheorem{theorem}{Theorem}[section]
\newtheorem{proposition}[theorem]{Proposition}
\newtheorem{lemma}[theorem]{Lemma}
\newtheorem{corollary}[theorem]{Corollary}
\newtheorem{Definition}[theorem]{Definition}

\newtheorem{Remark}[theorem]{Remark}
\newenvironment{remark}{\begin{Remark}\rm}{\end{Remark}}
\newtheorem{RHproblem}[theorem]{RH problem}

\newtheorem{Example}[theorem]{Example}
\newenvironment{example}{\begin{Example}\rm}{\end{Example}}

\newcommand{\C}{\mathbb{C}}
\newcommand{\D}{\mathbb D}

\newcommand{\N}{\mathbb{N}}

\newcommand{\R}{\mathbb{R}}

\newcommand{\T}{\mathbb{T}}
\newcommand{\Z}{\mathbb{Z}}

\newcommand{\OO}{\mathcal O}
\newcommand{\PP}{\mathcal P}

\newcommand{\eps}{\epsilon}

\def \deg{\mbox{{\rm deg}}}

\def\dim{\mathop{\mathrm{dim}}\nolimits}

\def \Poly{\mbox{{\rm Poly}}}

\def\Poly{\mathop{\mathrm{Poly}}\nolimits}

\renewcommand{\bar}{\overline}
\renewcommand{\tilde}{\widetilde}
\renewcommand{\hat}{\widehat}

\usepackage[bookmarksopen, naturalnames]{hyperref}
\begin{document}
\title{Polynomials associated to non-convex bodies}
\author{N. Levenberg and F. Wielonsky}

\maketitle 

\begin{abstract} Polynomial spaces associated to a convex body $C$ in $(\R^+)^d$ have been the object of recent studies. In this work, we consider polynomial spaces associated to non-convex $C$. We develop some basic pluripotential theory including notions of $C-$extremal plurisubharmonic functions $V_{C,K}$ for $K\subset \C^d$ compact. Using this, we discuss Bernstein-Walsh type polynomial approximation results and asymptotics of random polynomials in this non-convex setting.

\end{abstract}

\section{Introduction}  

Pluripotential theory is the study of plurisubharmonic (psh) functions. A fundamental result is that the extremal function
$$V_K(z):= \sup \{u(z):u\in L(\C^d), \ u\leq 0 \ \hbox{on} \ K\}$$
associated to a compact set $K\subset \C^d$, where $L(\C^d)$ is the Lelong class of all plurisubharmonic functions $u$ on $\C^d$ with the property that 
$$u(z) \leq \log^{+} |z|+c_{u}:=\max[0,\log |z|] +c_u,$$ 
where $c_{u}$ is a constant depending on $u$, may be obtained from the subclass of $L(\C^d)$ arising from polynomials:
$$V_K(z)=\max[0,\sup\{\frac{1}{deg(p)}\log |p(z)|:p \in \PP_{d}, \ \|p\|_K:=\max_{\zeta \in K}|p(\zeta)|\leq 1\}],$$
where $\PP_{d}=\C [z_1,...,z_d]$ denotes the family of all holomorphic polynomials of $d$ complex variables. This gives a connection with polynomial approximation; see Theorem \ref{BW} below. It is known that 
$V_K^*(z):=\limsup_{\zeta \to z} V_K(\zeta)$ is either a psh function in $L(\C^d)$ or else $V_K^*\equiv +\infty$ 
(this latter case occurs precisely when $K$ is pluripolar). We say $K$ is {\it regular} if $V_K=V_K^*$; i.e., $V_K$ is continuous.

In recent years, pluripotential theory associated to a convex body $C$ in $(\R^+)^d$ has been developed. Let $\R^+=[0,\infty)$ and fix a convex body $C\subset (\R^+)^d$ ($C$ is compact, convex and $C^o\not = \emptyset$). Associated with $C$ we consider the finite-dimensional polynomial spaces 
\begin{equation}\label{polync}\Poly(nC):=\{p(z)=\sum_{J\in nC\cap (\Z^+)^d}c_J z^J: c_J \in \C\}\end{equation}
for $n=1,2,...$ where $z^J=z_1^{j_1}\cdots z_d^{j_d}$ for $J=(j_1,...,j_d)$. For $C=\Sigma$ where 
\begin{equation}\label{simpl}
\Sigma:=\{(x_1,...,x_d)\in \R^d: 0\leq x_i \leq 1, \ \sum_{i=1}^d x_i \leq 1\},
\end{equation}
we have $\Poly(n\Sigma)$ is the usual space of holomorphic polynomials of degree at most $n$ in $\C^d$. For a nonconstant polynomial $p$ we define 
\begin{equation}\label{degp} \deg_C(p)=\min\{ n\in\N \colon p\in \Poly(nC)\}.\end{equation}
As in \cite{BBL}, \cite{BBLL}, \cite{BosLev}, we make the assumption on $C$ that 
\begin{equation}\label{sigmainkp} \Sigma \subset kC \ \hbox{for some} \ k\in \Z^+.  \end{equation}
Note that under hypothesis (\ref{sigmainkp}), we have $\cup_n \Poly(nC)=\PP_{d}$.

Recall the {\it indicator function} of a convex body $C$ is
\begin{equation}\label{phic}\phi_C(x_1,...,x_d):=\sup_{(y_1,...,y_d)\in C}(x_1y_1+\cdots x_dy_d).\end{equation}
We define the {\it logarithmic indicator function} of $C$ on $\C^d$
\begin{equation}\label{hc}H_C(z):=\sup_{ J\in C} \log |z^{ J}|:=\sup_{ J\in C} \log\left(|z_1|^{ j_1}\ldots |z_d|^{ j_d}\right) =\phi_C(\log|z_1|,\ldots,\log |z_d|)
\end{equation}
(the exponents $j_{k}$ need not be integers) and we use $H_C$ to define a generalization of $L(\C^d)$:
$$L_C=L_C(\C^d):= \{u\in PSH(\C^d): u(z) \leq H_C(z) +c_{u} \}.$$ 
Since $\phi_{\Sigma}(x_{1},\ldots,x_{d})=\max(x_{1}^{+},\ldots,x_{d}^{+})$, we have $L(\C^{d})=L_{\Sigma}$.
Moreover, it was stated in \cite{Bay} and shown in \cite{BHLP} that the {\it $C-$extremal function} 
$$V_{C,K}(z):=\sup \{u(z):u\in L_C(\C^d), \ u\leq 0 \ \hbox{on} \ K\}$$
of a compact set $K$ can be given as 
\begin{equation}\label{usecdef}V_{C,K} =\lim_{n\to \infty} \frac{1}{n} \log \Phi_n=\lim_{n\to \infty} \frac{1}{n} \log \Phi_{n,C,K}\end{equation}
pointwise on $\C^d$ where
$$ \Phi_n(z):= \sup \{|p(z)|: p\in \Poly(nC),  \ \|p\|_K\leq 1\}.$$
For $p\in \Poly(nC), \ n\geq 1$ we have $\frac{1}{n}\log |p|\in L_C$; hence we have a {\it Bernstein-Walsh inequality}
\begin{equation}\label{bwci} |p(z)|\leq \|p\|_Ke^{{\rm deg}_C(p)V_{C,K}(z)}.\end{equation}
We add that for $C$ satisfying (\ref{sigmainkp}), $K$ regular is equivalent to $V_{C,K}=V_{C,K}^*$ ($V_{C,K}$ is continuous); cf., \cite{BosLev}. 

The inequality (\ref{bwci}) leads to a connection with polynomial approximation. For $K$ a compact subset of $\C^{d}$, and $F\in C(K)$, a complex-valued, continuous function on $K$, we define
$$
d_{n}^{C}(F,K)  =\inf_{p\in\Poly(nC)}\|F-p\|_{K}.
$$
The following {\it Bernstein-Walsh type theorem} was proved in \cite{BosLev} to explain the use of various notions of degree for multivariate polynomials introduced by Trefethen in \cite{T}.
\begin{theorem}\label{BW}
Let $K$ be a compact subset of $\C^{d}$ with $V_{C,K}$ continuous and $F\in C(K)$. The following assertions are equivalent.
\\[5pt]
1) $\limsup_{n\to \infty }d_{n}^{C}(F,K)^{1/n}=1/R<1$;
\\[5pt]
2) the function $F$ is the restriction to $K$ of a function $\tilde F$ holomorphic on a domain $\Omega$ containing $K$, and $R$ is the largest real number such that
$$
\Omega_{C,R}=\{(z,w)\in\C^{2}:~V_{C,K}(z,w)<\log R\}\subset\Omega.
$$
\end{theorem}

This is a quantitative version of the Oka-Weil theorem: any $F$ holomorphic in a neighborhood of the polynomial hull
\begin{equation}\label{phull} \hat K:=\{z\in \C^d: |p(z)|\leq \|p\|_K, \ p\in \PP_{d}\}\end{equation}
of $K$ can be uniformly approximated on $\hat K$ by polynomials in $\PP_{d}$. Note that the ``smaller'' the convex body $C$, the sparser the collection $\Poly(nC)$ will be. For purposes of numerical analysis, sparseness is desirable. Indeed, 
provided $C\in (\R^+)^d$ is the closure of an open, connected set satisfying
\begin{equation} \label{stdhyp} \epsilon \Sigma \subset C \subset \delta \Sigma \ \hbox{for some} \ \delta > \epsilon >0, \end{equation} 
regardless of whether $C$ is convex, the finite-dimensional spaces $\Poly(nC)$ defined as in (\ref{polync}) make sense (as does the notion of $C-$degree in (\ref{degp})) and $\cup_n \Poly(nC)=\PP_d$. Thus, appealing to Oka-Weil, there is at least a possibility of a version of Theorem \ref{BW} in this setting. Examples of such $C$ are the $l^p$ balls  
\begin{equation}\label{ceepee} C=C_p=\{(x_1,...,x_d)\in (\R^+)^d: x_1^p+\cdots+ x_d^p\leq 1\}\end{equation}
with $0<p<1$. Numerical analysts even consider the limiting case of $p=0$,
\begin{equation}\label{ceez}
C_{0}=\cup_{j=1}^{d}\{(0,\ldots,x_{j},\ldots,0),~0\leq x_{j}\leq1\}
\end{equation}
(but see Example \ref{silly}). 

In this note, we begin a study of polynomial classes associated to non-convex $C$ satisfying (\ref{stdhyp}). The next section discusses general results on $C-$extremal functions $V_{C,K}$ defined in a fashion similar to (\ref{usecdef}). In section 3 we show that while one direction of Theorem \ref{BW} trivially generalizes, the other allows for interesting contrasts. 

These $C-$extremal functions $V_{C,K}$ are difficult to compute explicitly. However, in a probabilistic sense, one can ``generically'' recover them. Let $\tau$ be a probability measure on $K$ which is nondegenerate in the sense that $\|p\|_{\tau}:=\|p\|_{L^2(\tau)}=0$ for a polynomial $p$ implies $p\equiv 0$. Letting $\{p_j\}$ be an orthonormal basis in $L^2(\tau)$ for $\Poly(nC)$ constructed via Gram-Schmidt applied to a monomial basis $\{z^{\nu}\}$ of $\Poly(nC)$ we consider random polynomials of $C-$degree at most $n$ of the form 
$$H_n(z):=\sum_{j=1}^{m_n} a_{j}^{(n)}p_j(z)$$
where the $a_{j}^{(n)}$ are i.i.d. complex random variables and $m_n=$dim$(\Poly(nC)$). This places a probability measure $\mathcal H_n$ on $\Poly(nC)$. We form the product probability space of sequences of polynomials:
$$\mathcal H:=\otimes_{n=1}^{\infty} (\Poly(nC),\mathcal H_n).$$
The following was proved for $C=\Sigma$ in \cite{blrp} and for general convex $C$ in \cite{Bay}.

\begin{theorem} \label{point} Let $\tau$ be a probability measure on $K$ such that $(K,\tau)$ satisfies a Bernstein-Markov property and let $a_{j}^{(n)}$ be i.i.d. complex random variables having distribution $\phi(z)dm_2(z)$ where $dm_2$ denotes Lebesgue measure on $\R^2=\C$. Assume for some $T>0$, 
\begin{equation}\label{hyp1} |\phi(z)|\leq T \ \hbox{for all} \ z\in \C; \ \hbox{and} \end{equation}
\begin{equation}\label{hyp2} 
\left|\int_{|z|\geq R}\phi(z)dm_2(z)\right|\leq T/R^2 \ \hbox{for all} \ R \ \hbox{sufficiently large}.\end{equation}
Then almost surely in $\mathcal H$ we have
$$\bigl(\limsup_{n\to \infty}\frac{1}{n}\log |H_n(z)|\bigr)^*=V_{C,K}^*(z),\quad z\in\C^{d}.$$
\end{theorem}

\noindent The {\it Bernstein-Markov property} will be defined in the next section.

In section 4 we give a version of Theorem \ref{point} in the nonconvex setting. We conclude in section 5 with some open questions.

\section{Non-convex preliminaries} Let $C$ be the closure of an open, connected set satisfying (\ref{stdhyp}). For simplicity, we take $\delta =1$ in (\ref{stdhyp}). As noted in the introduction, the definitions of the vector spaces
\begin{align*}
\Poly(nC)=\{p(z) & =\sum_{J\in nC \cap \N^d}c_{J}z^{J}=\sum_{J\in nC \cap \N^d}c_{J}z_1^{j_1}\ldots z_d^{j_d},~c_{J}\in\C\},\quad n=1,2,\ldots
\end{align*}
and, for a nonconstant polynomial $p$, the $C-$degree 
$$ \deg_C(p)=\min\{ n\in\N\colon p\in \Poly(nC)\},$$
can be defined as in the convex case. However, if $C$ is not convex, two vital ingredients are lacking:
\begin{enumerate}
\item $\Poly(nC)\cdot \Poly(mC)$ may {\it not} be contained in $\Poly(n+m)C$; and 
\item there is no good analogue/replacement for the logarithmic indicator function $H_C$ in (\ref{hc}) and hence the $L_C$ Lelong classes.
\end{enumerate}
Item 1. is crucial in proving 2) implies 1) in Theorem \ref{BW}. To explain 2., for $C$ the closure of an open, connected set satisfying (\ref{stdhyp}), using (\ref{phic}) and (\ref{hc}) yields $H_C=H_{co(C)}$ where $co(C)$ denotes the convex hull of $C$. If, e.g., $C=C_p$ in (\ref{ceepee}) with $0<p<1$, then $H_{C_p}=H_{\Sigma}$. 

\subsection{$C-$extremal function for non-convex $C$}

Given a compact set $K\subset \C^d$, we will define a $C-$extremal function using the $\Poly(nC)$ classes: for $n=1,2,...$, let
\begin{equation}\label{Phin} \Phi_n(z)=\Phi_{n,C,K}(z):=\sup \{|p(z)|: p\in \Poly(nC), \ \|p\|_K\leq 1\}\end{equation}
(note taking $p\equiv 1$ shows $\Phi_n(z)\geq 1$) and
\begin{equation}\label{def-V}
V_{C,K}(z):=\limsup_{n\to \infty} \frac{1}{n}\log \Phi_n(z).
\end{equation}
For a polynomial $p\in \Poly(nC)$ we have $\deg_C(p)\leq n$. Thus 
$$\Phi_n(z)^{{1}/{n}}\leq \Phi_n(z)^{{1}/{{\rm deg}_C(p)}}.$$ 
This shows that 
\begin{equation}\label{vgevc}V_{C,K}(z)\leq \sup\{\frac{1}{{\rm deg}_C(p)}\log |p(z)|:p\in \PP_{d}, \ \|p\|_K\leq 1\}.\end{equation}
We do not know if equality holds (in general) in (\ref{vgevc}). From (\ref{stdhyp}), 
$$\epsilon V_K \leq V_{C,K}\leq V_K$$
so that for $K$ nonpluripolar, $V_{C,K}^*$ is a plurisubharmonic function (indeed, $V_{C,K}^*\in L(\C^d)$). Furthermore, 
$$C\subset C' \ \hbox{implies} \ V_{C,K}\leq V_{C',K}$$ which follows from the facts that  
$\Poly(nC)\subset \Poly(nC')$ and $p\in \Poly(nC)$ implies $\deg_C(p)\geq \deg_{C'}(p)$.

\begin{remark} From the definition (\ref{def-V}), for $K\subset \C^d$ compact we have $V_{C,K}=V_{C,\hat K}$ where $\hat K$ is the polynomial hull of $K$ (recall (\ref{phull})). Let $C$ be the closure of an open, connected set satisfying (\ref{stdhyp}). We remark that 
\begin{equation}\label{polyhull} Z(K):=\{z\in \C^d: V_{C,K}(z)=0\}=\hat K. \end{equation} To see this, since $V_{C,K}\geq 0$ we clearly have $\hat K \subset Z(K)$. For the reverse inclusion, if $z_0\not \in \hat K$ there exists a polynomial $p$ with 
$$\frac{|p(z_0)|}{\|p\|_K}=:1+\lambda>1.$$ 
By (\ref{stdhyp}), for each positive integer $k$, $p^k\in Poly(n_kC)$ for some positive integer $n_k={\rm deg}_C(p^k)$ with $n_k\uparrow \infty$ and 
$$
n_{k}\leq\deg_{\eps\Sigma}(p^{k})\leq k~\deg_{\eps\Sigma}(p).
$$
Thus 
$$\limsup_{k\to \infty} \frac{1}{n_k} \log \frac{|p^k(z_0)|}{\|p^k\|_K}\geq 
\limsup_{k\to \infty} \frac{1}{ k \cdot {\rm deg}_{\eps\Sigma}(p)}\log (1+\lambda)^k
=\frac{1}{ {\rm deg}_{\eps\Sigma}(p)}\log (1+\lambda)>0$$
so that $V_{C,K}(z_0)>0$. \end{remark}

One class of compact sets $K$ for which $C-$extremal functions can be computed are products of planar compacta. 
For $E\subset \C$ a planar compacta, let $g_{E}$ be the classical Green function of (the unbounded component of the complement of) $E$.

\begin{proposition}\label{easy} Let $C\subset \Sigma$ be a connected set containing $C_{0}$ in (\ref{ceez}). For $K=E_1\times \cdots \times E_d$, a product of regular, planar compacta $E_j$, we have
$$V_{C,K}(z)=V_{\Sigma,K}(z)=\max_{j=1,...,d} g_{E_j}(z_j).$$
In particular, this holds for $C$ the closure of an open, connected set satisfying (\ref{stdhyp}). 
\end{proposition}

\begin{proof} It is classical that $V_{\Sigma,K}(z)=\max_{j=1,...,d} g_{E_j}(z_j)$ (cf., Theorem 5.1.8 \cite{K}). Since $C\subset \Sigma$, clearly 
$V_{C,K}(z)\leq V_{\Sigma,K}(z)$. For the reverse inequality, for $n=1,2,...$ define
$$\Psi_n(z):=\sup\{|p(z)|:~p(z)=p_1(z_1)+\cdots +p_d(z_d),~\deg~ p_i \leq n,~\|p\|_{K}\leq 1\}
$$
(here $\deg=\deg_{\Sigma}$) and 
$$V(z):=\limsup_{n\to \infty} \frac{1}{n}\log \Psi_n(z).$$
Since $C_{0} \subset C$, $\Psi_n(z)\leq \Phi_n(z)$ in (\ref{Phin}) and hence $V(z)\leq V_{C,K}(z)$. Fixing $j\in \{1,...,d\}$ and taking $p_k\equiv 0$ for $k\not =j$, we have
$$\Psi_n(z)\geq \sup\{|p_j(z_j)|: \deg~ p_j \leq n, \ \|p_j\|_{E_j}\leq 1\}.$$
Thus
$$V(z)\geq \limsup_{n\to \infty} \bigl(\frac{1}{n}\log \sup\{|p_j(z_j)|: \deg~ p_j \leq n, \ \|p_j\|_{E_j}\leq 1\}\bigr)=g_{E_j}(z_j).$$
This holds for $j=1,...,d$ and the proposition is proved.

\end{proof}

\begin{remark}\label{imp} Since, in this setting, 
$$\sup\{|p_{j}(z_{j})|:~\deg p_{j}\leq n,~\|p_{j}\|_{E_{j}}\}\leq\Psi_n(z)\leq \Phi_n(z),$$ 
and
$$
\lim_{n\to\infty}\frac1n\log\sup\{|p_{j}(z_{j})|:~\deg p_{j}\leq n,~\|p_{j}\|_{E_{j}}\}=g_{E_{j}}(z_{j})
$$
we have a true limit in (\ref{def-V}). Moreover, the proof of Proposition \ref{easy} shows that for $K$ compact in $\C^d$, if we 
let $E_j = \pi_j (K)$ where $\pi_j$ is the projection from $z=(z_1,...,z_d)$ to the $j-$th coordinate $z_j$, then 
$$V_{\Sigma,K}^*(z) \geq V_{C,K}^*(z) \geq \max_{j=1,...,d} g_{E_j}(z_j).$$
We mention that for $C$ convex and $K=E_1\times \cdots \times E_d$, a product of regular, planar compacta $E_j$, we have (cf., \cite[Prop. 2.4]{BosLev})
\begin{equation}\label{tabform}V_{C,K}(z_1,\ldots,z_d)= \phi_C(g_{E_1}(z_1),\ldots,g_{E_d}(z_d))\end{equation}
where $\phi_C$ is defined in (\ref{phic}).

\end{remark}

Examples of nonconvex $C$ satisfying the hypotheses of Proposition \ref{easy} are the $l^p$ balls   
$$C_p=\{(x_1,\ldots,x_d)\in (\R^+)^d: x_1^p+\cdots + x_d^p\leq 1\}$$
from (\ref{ceepee}) with $0<p<1$. Note Proposition \ref{easy} is also valid in the limiting case $p=0$. 

\subsection{$L^2-$approach to $V_{C,K}$} We next discuss several ways of recovering $V_{C,K}$ in this non-convex setting, motivated by the standard (and convex) settings. Often $L^2-$norms are more convenient to work with than $L^{\infty}-$norms. To this end, for $K$ a compact set in $\C^d$ and $\tau$ a positive Borel measure on $K$, we say that $(K,\tau)$ satisfies a {\it Bernstein-Markov property} if 
for any polynomial $p_{n}$ of degree $n$ and any $n$
\begin{equation} \label{wtdbm}\|p_n\|_K\leq M_n \|p_n\|_{\tau} \ \hbox{where} \ \limsup_{n\to \infty} M_n^{1/n}=1.\end{equation} From our hypothesis (\ref{stdhyp}), this is equivalent to (\ref{wtdbm}) for $p_n \in \Poly(nC)$. For simplicity, we assume $\tau(K)=1$.

In the standard pluripotential setting of $C=\Sigma$, let 
$$m_n=m_n(\Sigma)=\dim(\Poly(n\Sigma))= {d+n\choose n}=\OO(n^d).$$ 
We take a lexicographical ordering of the monomials $\{z^{\nu}\}_{|\nu|\leq n}$ in $\Poly(n\Sigma)$ and write these as $\{e_j(z)\}_{j=1}^{m_n}$. Let $\{p_j\}_{j=1,...,m_n}$ be a set of orthonormal polynomials of degree at most $n$ in $L^2(\tau)$ gotten by applying the Gram-Schmidt process to these monomials in $\Poly(n\Sigma)$. For each $n=1,2,...$ consider the corresponding Bergman kernel
$$S_n(z,\zeta):=\sum_{j=1}^{m_n} p_j(z)\overline{p_j(\zeta )}$$
and the restriction to the diagonal
\begin{equation}\label{snfcn}S_n(z,z)=\sum_{j=1}^{m_n} |p_j(z)|^2.\end{equation}
By the reasoning in \cite{BS}, we have the following.
\begin{proposition}\label{snfcnprop} Let $K\subset \C^d$ be compact and nonpluripolar and let $\tau$ be a probability measure on $K$ such that $(K,\tau)$ satisfies (\ref{wtdbm}). Then with $S_n(z,z)$ defined in (\ref{snfcn}), 
$$\lim_{n\to \infty} \frac{1}{2n}\log S_n(z,z) = V_{\Sigma,K}(z),\quad z\in\C^{d}.$$
If $V_{\Sigma,K}$ is continuous, the convergence is uniform on compact subsets of $\C^d$. 
\end{proposition}
\begin{proof}[Sketch of proof]
We briefly indicate the two main steps in the proof since these will be generalized. First, for each $n=1,2,...$ define
$$ \phi_n(z):=\sup \{ |p(z)|: p\in  \Poly(n\Sigma), \ \|p\|_K\leq 1\}.$$
Then, from \cite[Theorem 5.1.7]{K},
\begin{equation}\label{phin} \lim_{n\to \infty} \frac{1}{n}\log \phi_n(z) = V_{\Sigma,K}(z)\end{equation} 
pointwise on $\C^d$; and the convergence is uniform on compact subsets of $\C^d$ if $V_{\Sigma,K}$ is continuous. This is \cite{BS}, Lemma 3.4. The next step is a comparison between $\phi_n(z)$ and $S_n(z,z)$: 
\begin{equation} \label{step2} 1 \leq \frac{S_n(z,z)}{\phi_n(z)^2}\leq M_n^{2}m_n, \ z\in \C^d, \end{equation}
where $M_n$ is as in (\ref{wtdbm}). The left-hand inequality follows from the reproducing property of the Bergman kernel $S_n(z,\zeta)$ and the Cauchy-Schwarz inequality and is valid for any $\tau$ for which one has an orthonormal basis in $L^2(\tau)$ for $\Poly(n\Sigma$). Let $p$ be a polynomial of degree at most $n$ with $\|p\|_K\leq 1$. Writing $p(z)=\sum_{j=1}^{m_n}a_j p_j(z)$, 
$$|p(z)|^2\leq \sum_{j=1}^{m_n}|a_j|^2 \cdot \sum_{j=1}^{m_n}|p_j(z)|^2 = \|p\|_{\tau}^2\cdot \sum_{j=1}^{m_n}|p_j(z)|^2 \leq \|p\|_K^2\cdot \sum_{j=1}^{m_n}|p_j(z)|^2\leq S_n(z,z).$$ 
Since $\phi_n(z)=\sup \{ |p(z)|: p\in  \Poly(n\Sigma), \ \|p\|_K\leq 1\}$, taking the supremum over all such $p$ gives the left-hand inequality. The right-hand inequality uses the Bernstein-Markov property of $(K,\tau)$. We have $\|p_j\|_K \leq M_n$ so that $|p_j(z)|/M_n\leq \phi_n(z)$ and
$$S_n(z,z)=\sum_{j=1}^{m_n}|p_j(z)|^2\leq m_n\cdot M_n^2\cdot [\phi_n(z)]^2.$$
\end{proof}
The exact same proof is valid for $C$ convex satisfying (\ref{sigmainkp}) (cf., Proposition 2.11 of \cite{Bay}). We note that the analogue of (\ref{phin}) in this setting uses the fact that 
$$\Poly(nC)\cdot \Poly(mC)\subset \Poly(n+m)C.$$
Given $C$, the closure of an open, connected set satisfying (\ref{stdhyp}) which contains $C_{0}$ in (\ref{ceez}), if we know for a given compact set that we have the pointwise limit
\begin{equation}\label{forPhin} \lim_{n\to \infty} \frac{1}{n}\log \Phi_n(z) =V_{C,K}(z), \end{equation}
i.e., the limit exists and equals $V_{C,K}(z)$, then the analogue of Proposition \ref{snfcnprop} holds in this setting, except perhaps for the local uniform convergence. We state this as Proposition \ref{touse}. Moreover, we get convergence in $L^1_{loc}(\C^d)$ as well. Here we use an ordering $\prec_C$ on $\N^d$ which respects $\deg_C(p)$ in the sense that $\alpha \prec_C \beta$ whenever $\deg_C(z^{\alpha}) < \deg_C(z^{\beta})$, and 
$$S_n^C(z,z)=\sum_{j=1}^{m_n} |p_j(z)|^2$$
where $\{p_j\}_{j=1,...,m_n}$ is an orthonormal basis in $L^2(\tau)$ for $\Poly(nC)$ constructed via Gram-Schmidt applied to an ordered monomial basis $\{z^{\nu}\}$ of $\Poly(nC)$. Here 
$$m_n=m_n(C)=\hbox{dim}(\Poly(nC)).$$ For the $L_{loc}^1(\C^d)$ convergence, we will use the following standard result; this will also be needed in section 4. The proof is identical to that of Proposition 4.4 of \cite{blrp}.

\begin{proposition} \label{thmdet} Let $\{\psi_n\}$ be a locally uniformly bounded above family of plurisubharmonic functions on $\C^d$. Suppose for any subsequence $J$ of positive integers we have 
$$\bigl(\limsup_{n\in J}\psi_n(z)\bigr)^*=V(z)$$
for all $z\in \C^d$ where $V\in PSH(\C^d)\cap L^{\infty}_{loc}(\C^d)$. Then $\psi_n \to V$ in $L_{loc}^1(\C^d)$.
\end{proposition}

\begin{proposition}\label{touse} Let $C$ be the closure of an open, connected set satisfying (\ref{stdhyp}). Let $K\subset \C^d$ be compact, nonpluripolar and satisfying (\ref{forPhin}). Finally, let $\tau$ be a positive Borel measure on $K$ such that $(K,\tau)$ satisfies (\ref{wtdbm}). Then the sequences $\{{1\over 2n}\log S^C_n\}$ and $\{{1\over n}\log \Phi_n\}$ are locally uniformly bounded above and 
$$ \lim_{n\to \infty} \frac{1}{2n}\log S_n^C(z,z) = V_{C,K}(z)$$
pointwise on $\C^d$. Furthermore, both sequences $\{{1\over 2n}\log S^C_n\}$ and $\{{1\over n}\log \Phi_n\}$ converge to $V_{C,K}^*$ in $L^1_{loc}(\C^d)$.
 
\end{proposition}

\begin{proof} The analogue of (\ref{step2}) holds with $S_n, \phi_n, m_n=m_n(\Sigma)$ replaced by $S_n^C, \Phi_n, m_n=m_n(C)$ with the exact same proof:
\begin{equation}\label{step2b} 1 \leq \frac{S^C_n(z,z)}{\Phi_n(z)^2}\leq M_n^{2}m_n.\end{equation}
Under the hypothesis (\ref{forPhin}), the pointwise convergence of $\frac{1}{2n}\log S_n^C(z,z)$ to  $V_{C,K}(z)$ follows. Moreover, from (\ref{step2b}) and $C\subset \Sigma$
\begin{equation}\label{unifbd}\frac{1}{2n}\log S_n^C(z,z) \leq \frac{1}{n}\log \Phi_n(z)+ \frac{1}{2n}\log (M_n^{2}m_n)\leq V_K^*(z)+\limsup_{n\to \infty}\frac{1}{2n}\log (M_n^{2}m_n)=V_K^*(z)\end{equation}
which shows the sequences $\{{1\over 2n}\log S^C_n\}$ and $\{{1\over n}\log \Phi_n\}$ are locally uniformly bounded above. Proposition \ref{thmdet} immediately shows
$$\frac{1}{2n}\log S_n^C \to V_{C,K}^* \ \hbox{in} \ L^1_{loc}(\C^d).$$
Finally, for each $n$, the function ${1\over n}\log \Phi_n^*$ is psh and is equal to ${1\over n}\log \Phi_n$ except perhaps for a pluripolar set. Since a countable union of pluripolar sets is pluripolar,
$$\lim_{n\to \infty} {1\over n}\log \Phi_n^*(z)=\lim_{n\to \infty} {1\over n}\log \Phi_n (z)=V_{C,K}(z)$$
outside of a pluripolar set. Hence $[\lim_{n\to \infty} {1\over n}\log \Phi_n^*(z)]^*=V_{C,K}^*(z)$ for all $z\in \C^d$. By Proposition \ref{thmdet}, ${1\over n}\log \Phi_n^*\to V_{C,K}^* \ \hbox{in} \ L^1_{loc}(\C^d)$ and hence the same is true for  $\{{1\over n}\log \Phi_n\}$. \end{proof}

From Remark \ref{imp}, the full conclusion of Proposition \ref{touse} holds if $C\subset \Sigma$ contains $C_{0}$ in (\ref{ceez}) and $K=E_1\times \cdots \times E_d$, a product of regular, planar compacta $E_j$. In fact, we get slightly more, namely local uniform convergence, for certain Bernstein-Markov measures on $K$. Suppose $\mu_j$ is a Bernstein-Markov measure on $E_j$ for $j=1,...,d$. Then $\mu:=\otimes_{j=1}^d \mu_j$ is a Bernstein-Markov measure on $K$. If $\{e^{(j)}_k\}$ is an orthonormal basis of polynomials for $L^2(\mu_{E_j})$ with $deg(e^{(j)}_k)=k$, then we know that 
$$\lim_{n\to \infty} \frac{1}{2n} \log   \sum_{k=0}^n |e^{(j)}_k(z_j)|^2 =g_{E_j}(z_j)$$
locally uniformly on $\C=\C_{z_j}$. Then, we have
\begin{proposition}\label{imp2}
The $n-$th Bergman function 
$$B_n^{\mu,C}(z):=S_n^{C}(z,z)$$ for $\mu$ associated to the vector space $\Poly(nC)$ satisfies
\begin{equation}\label{locunif} 
\lim_{n\to \infty} \frac{1}{2n} \log B_n^{\mu,C}(z_1,...,z_d)= \max_{j=1,...,d}g_{E_j}(z_j)
\end{equation}
locally uniformly on $\C^d$ (and hence in $L^1_{loc}(\C^d)$). 
\end{proposition}
\begin{proof}
We first consider the $n-$th Bergman function $B_n^{\mu,C_0}(z):=S_n^{C_0}(z,z)$ for 
$$\Poly(nC_0)=\hbox{span}\{1,z_1,...,z_d,z_1^2,...,z_d^2,...,z_1^n,...,z_d^n\}$$ 
where $C_0$ is the $p=0$ case of the $l^p$ ball $C_p$ in (\ref{ceez}). Then, assuming $\mu$ is a probability measure, we have 
$$d-1+B_n^{\mu,C_0}(z_1,...,z_d)= d+\sum_{j=1}^d\sum_{k=1}^n |e^{(j)}_k(z_j)|^2
=\sum_{j=1}^d\sum_{k=0}^n |e^{(j)}_k(z_j)|^2.$$
Now recall that for real $a_1,...,a_d$, 
\begin{equation}\label{estim}\lim_{n\to \infty} \frac{1}{2n}\log \left(e^{2na_1}+\cdots +e^{2na_d}\right)= \max_{j=1,...,d} a_j.\end{equation}
Taking sequences $\{a^{(j)}_n(z_j):=  \sum_{k=0}^n |e^{(j)}_k(z_j)|^2\}$ so that 
$$\lim_{n\to \infty}\frac{1}{2n} \log a^{(j)}_n(z_j) = a_j(z_j):=g_{E_j}(z_j), \ j=1,...,d$$
where the convergence is locally uniform in each $\C=\C_{z_j}$, using (\ref{estim}) we have
$$\lim_{n\to \infty} \frac{1}{2n}\log \left(\sum_{j=1}^d a^{(j)}_n(z_j)\right)= \max_{j=1,...,d} a_j(z_j)
$$
locally uniformly in $\C^d$, and thus
$$
\lim_{n\to \infty} \frac{1}{2n} \log B_n^{\mu,C_0}(z_1,...,z_d)
= \lim_{n\to \infty} \frac{1}{2n} \log (d-1+B_n^{\mu,C_0}(z_1,...,z_d))
=\max_{j=1,...,d}g_{E_j}(z_j)
$$
locally uniformly in $\C^d$, which is (\ref{locunif}) for $C=C_0$. The case $C=C_1=\Sigma$ follows from the more general Proposition \ref{snfcnprop}. For other $C$ as in Proposition \ref{easy}, $C_0\subset C \subset C_1$ which implies the inequality 
$$ B_n^{\mu,C_0}\leq B_n^{\mu,C} \leq B_n^{\mu,C_1}$$
and hence the general case of (\ref{locunif}).
\end{proof}

\begin{remark} Unlike the case where $C$ is convex, in the non-convex setting, it is unclear whether one has 
$$\lim_{n\to \infty} \frac{1}{n}\log \Phi_n(z) =V_{C,K}(z)$$
pointwise; and, even if this holds and $V_{C,K}$ is continuous, it is unclear whether the limit is locally uniform.  From Proposition \ref{imp2} and (\ref{step2b}), for $C$ the closure of an open, connected set satisfying (\ref{stdhyp}) which contains $C_{0}$ in (\ref{ceez}), all of these properties hold for $K=E_1\times \cdots \times E_d$ a product of regular, planar compacta $E_j$. 

\end{remark}

More generally, let $\mu$ be any positive measure on $K$ such that one can form orthonormal polynomials $\{p_{\alpha}\}$ using Gram-Schmidt on the monomials $\{z^{\alpha}\}$. As before we use an ordering $\prec_C$ on $\N^d$ which respects $\deg_C(p)$. The following argument of Zeriahi \cite{zer} is valid in this setting.

\begin{proposition} \label{threeone} Let $K\subset \C^d$ be compact and nonpluripolar and let $C$ be the closure of an open, connected set satisfying (\ref{stdhyp}). Then 
\begin{equation}\label{Zer}
\limsup_{|\alpha|\to \infty} \frac{1}{{\rm deg}_C(p_{\alpha})}\log|p_{\alpha}(z)|\geq V_{C,K}(z), \quad z\not\in \hat K.
\end{equation}
\end{proposition}  

\begin{proof} Let $Q_n \in \Poly(nC)$ and $\|Q_n\|_K\leq 1$. From the property of the ordering $\prec_C$, we can write $Q_n=\sum_{\alpha \in nC}c_{\alpha}p_{\alpha}$. Then
$$|c_{\alpha}|=|\int_K Q_n\bar p_{\alpha} d\mu| \leq \int_K |\bar p_{\alpha}| d\mu\leq \sqrt{\mu(K)}$$
by Cauchy-Schwarz. Hence 
$$|Q_n(z)|\leq m_n \sqrt{\mu(K)} \max_{\alpha \in nC}|p_{\alpha}(z)|$$ where $m_n=$dim$(\Poly(nC))$.

Fix $z_0\in \C^d\setminus \hat K$ and let $\alpha_n\in nC$ be a multiindex with $\deg_C(p_{\alpha_n})$ largest such that 
$$|p_{\alpha_n}(z_0)|=\max_{\alpha \in nC}|p_{\alpha}(z_0)|.$$
We claim that taking any such sequence $\{\alpha_n=\alpha_n(z_0)\}_{n=1,2,...}$, 
$$\lim_{n\to \infty} \deg_C(p_{\alpha_n})=+\infty.$$
For if not, then by the above argument, there exists $A< \infty$ such that for any $n$ and any $Q_n\in \Poly(nC)$ with $\|Q_n\|_K\leq 1$,
$$|Q_n(z_0)|\leq m_n \sqrt{\mu(K)} \max_{\deg_C(p_{\alpha})\leq A }|p_{\alpha}(z_0)|=m_n M(z_0)$$
where $M(z_0)$ is independent of $n$. But then $\Phi_n(z_0)\leq m_n  M(z_0)$ so that
$$V_{C,K}(z_0)\leq \limsup_{n\to \infty} [\frac{1}{n}\log m_n +\frac{1}{n}\log M(z_0)]=0$$
which contradicts $z_0 \in \C^d\setminus \hat K$ from (\ref{polyhull}). We conclude that for any $z \in \C^d\setminus \hat K$, for any $n$ and any $Q_n\in \Poly(nC)$ with $\|Q_n\|_K\leq 1$,
$$\frac{1}{n}\log|Q_n(z)|\leq \frac{1}{n}\log m_n +\frac{1}{n} \log |p_{\alpha_n}(z)|$$
where we can assume $\deg_C(p_{\alpha_n})\uparrow \infty$. 
Thus, for such $z$, 
\begin{align*}
V_{C,K}(z) & \leq \limsup_{n \to \infty} \frac{1}{n}\log|p_{\alpha_n}(z)| \leq \limsup_{n \to \infty} \frac{1}{{\rm deg}_C(p_{\alpha_n})}\log|p_{\alpha_n}(z)|\\
& \leq \limsup_{|\alpha|\to \infty} \frac{1}{{\rm deg}_C(p_{\alpha})}\log|p_{\alpha}(z)|
\end{align*}
where we have used  $\deg_C(p_{\alpha_n})\leq n$.

\end{proof}

\begin{corollary}\label{vvc} Let $K\subset \C^d$ be compact and nonpluripolar and let $C$ be the closure of an open, connected set satisfying (\ref{stdhyp}). Then for any Bernstein-Markov measure $\mu$ for $K$, 
$$V_{C,K}(z)=\limsup_{|\alpha|\to \infty} \frac{1}{{\rm deg}_C(p_{\alpha})}\log|p_{\alpha}(z)|, \quad z\not\in \hat K.$$

\end{corollary}

\begin{proof} 
In view of Proposition \ref{threeone}, it is sufficient to prove that $V_{C,K}(z)$, $z\not \in\hat K$, is larger than the left-hand side of (\ref{Zer}). 
Assuming $\mu$ to be a probability measure, for the orthonormal polynomials $\{p_{\alpha}\}$, we have
$$1\leq\|p_{\alpha}\|_K \leq M_{{\rm deg}_C(p_{\alpha})} \ \hbox{and} \ M_{{\rm deg}_C(p_{\alpha})}^{1/{\rm deg}_C(p_{\alpha})}\to 1$$
using (\ref{stdhyp}) and the Bernstein-Markov property (\ref{wtdbm}). 
Thus
$$\lim_{|\alpha|\to \infty} \|p_{\alpha}\|_K^{1/{\rm deg}_C(p_{\alpha})}=1.$$
Consequently
$$
V_{C,K}(z)\geq\limsup_{|\alpha|\to \infty} \frac{1}{{\rm deg}_C(p_{\alpha})}\log\frac{|p_{\alpha}(z)|}
{\|p_{\alpha}\|_{K}}=
\limsup_{|\alpha|\to \infty} \frac{1}{{\rm deg}_C(p_{\alpha})}\log{|p_{\alpha}(z)|}.
$$

\end{proof}

\noindent By Proposition 3.1 of \cite{BLPW}, for any compact set $K\subset \C^d$ there exists a Bernstein-Markov measure $\mu$ for $K$. We will use Corollary \ref{vvc} in the next subsection.

\subsection{$K=B$, the complex Euclidean ball}

We first remark that some version of assumption (\ref{stdhyp}) seems natural in order that we have $\cup_n \Poly(nC) = \PP_{d}$. Moreover, if $\cup_n \Poly(nC) \varsubsetneq  \PP_{d}$, equality in (\ref{polyhull}) may fail. Indeed, let $C=C_0$ and consider $K=B:=\{z\in \C^d: |z_1|^2+ \cdots +|z_d|^2\leq 1\}$, the complex Euclidean ball in $\C^d$. Let $\mu$ be normalized surface area measure on $\partial B$ and let $\nu$ be normalized Haar measure on the torus 
$T:=\{z\in \C^d: |z_1|= \cdots = |z_d| = 1\}$. The monomials $z^{\alpha}=z_1^{\alpha_1}\cdots z_d^{\alpha_d}$ are orthonormal with respect to $\nu$ while the monomials $z_j^{\alpha_j}$ are orthogonal with respect to $\mu$ with 
$$a_{k}:=\|z_j^{k}\|^2_{\mu}=\frac{(d-1)!k!}{(d-1+k)!},\qquad j=1,\ldots,d.$$ 
We consider the $n-$th Bergman functions $B_n^{\mu,C_0}(z)$ and $B_n^{\nu,C_0}(z)$ for 
$$\hbox{Poly}(nC_0)=\hbox{span}\{1,z_1,...,z_d,z_1^2,...,z_d^2,...,z_1^n,...,z_d^n\}.$$ 
We have
$$B_n^{\mu,C_0}(z)= 1+a_{1}^{-1}(|z_1|^2+\cdots +|z_d|^2)+a_{2}^{-1}(|z_1|^4+\cdots +|z_d|^4)+\cdots+ a_{n}^{-1}(|z_1|^{2n}+\cdots +|z_d|^{2n})$$
and
$$B_n^{\nu,C_0}(z)= 1+(|z_1|^2+\cdots +|z_d|^2)+(|z_1|^4+\cdots +|z_d|^4)+\cdots +(|z_1|^{2n}+\cdots +|z_d|^{2n}).$$
Thus
\begin{equation}\label{numu} B_n^{\nu,C_0}(z)\leq B_n^{\mu,C_0}(z)\leq 
1+a_{n}^{-1}[B_n^{\nu,C_0}(z)-1].
\end{equation}
Similar to the proof of Proposition \ref{imp2}, we have
$$\lim_{n\to \infty} {1\over 2n}\log B_n^{\nu,C_0}(z)=\max [0,\log|z_1|,...,\log |z_d|]$$
locally uniformly for all $z\in \C^d$. From (\ref{numu}) we also have 
$$\lim_{n\to \infty} {1\over 2n}\log B_n^{\mu,C_0}(z)=\max [0,\log|z_1|,...,\log |z_d|]$$
locally uniformly for all $z\in \C^d$. However, the inequality (\ref{step2}) is valid for $C_0,B$ and $\mu$; in the above notation, since $\hbox{dim}(\hbox{Poly}(nC_0))=dn+1$,
$$1 \leq {B_n^{\mu,C_0}(z) \over \Phi_n(z)^2}\leq M_n^{2}(dn+1)$$
where
$$\Phi_n(z)=\Phi_{n,C_0,B}(z):=\sup \{|p(z)|: p\in \hbox{Poly}(nC_0), \ \|p\|_B\leq 1\}.$$
Hence $\lim_{n\to \infty} \frac{1}{n}\log \Phi_n(z)$ exists and equals $\max [0,\log|z_1|,...,\log |z_d|]$ as well. 

This shows 
\begin{equation}\label{vcob} V_{C_0,B}(z)=\max \left(0,\log|z_1|,...,\log |z_d|\right).\end{equation} 
In particular, 
$$Z(B)=\{z:V_{C_0,B}(z)=0\}=\{z: \max_{j=1,...,d} |z_j|\leq 1\}$$
so that $B =\hat B \varsubsetneq Z(B)$ when $d>1$. 

Moreover, given $z_0=(z_{0,1},...,z_{0,d})\not \in B$, writing $z_{0,j}=|z_{0,j}|e^{i\phi_j}$ and defining $p(z):=\sum_{j=1}^d e^{-2i\phi_j}z_{j}^2$, we have  
$$\|p\|_B\leq \max_{z \in B} \left(|z_1|^2+\cdots +|z_d|^2\right)=1$$
while
$$|p(z_0)|=|z_{0,1}|^2+\cdots +|z_{0,d}|^2>1.$$
Since $p\in \Poly(2C_0)$, this shows that if one defines
$$\tilde V_{C_0,B}(z):= \sup\{{1\over {\rm deg}_{C_0}(p)}\log |p(z)|:p\in \cup_n \hbox{Poly}(nC_0), \ \|p\|_B\leq 1\},$$
then $\tilde V_{C_0,B}(z)>0$ for all $z \not \in B$ so that $\tilde V_{C_0,B}\not = V_{C_0,B}$. Thus equality fails to hold in (\ref{vgevc}) for $C=C_0$ and $K=B$. 

Next we show that, unlike the case of product sets $K=E_1\times \cdots \times E_d$ in Proposition \ref{easy}, for $K=B$, the $C_p-$extremal functions for $0<p<1$ do not coincide with the $C_0-$ and $C_1=\Sigma -$extremal functions.

\begin{proposition} \label{propb} For $0<p<1$, we have 
$$V_{C_{0},B}(z)<V_{C_p,B}(z) < V_{C_{1},B}(z) $$
at certain points $z\in \C^d$.
\end{proposition}

\begin{proof} For simplicity we let $d=2$ and use variables $(z,w)$. As shown above, 
$$V_{C_0,B}(z,w)=\max [0,\log|z|, \log |w|]=0< V_{C_p,B}(z,w)$$
for $(z,w)\in Z(B)\setminus B\not = \emptyset$ and $0<p< 1$. We next verify for $0<p<1$ that 
$$V_{C_p,B}(z,w) \not = V_{C_1,B}(z,w)$$ 
for certain points. We utilize Corollary \ref{vvc}: taking $K=B\subset \C^2$, $\mu$ normalized surface area measure on $\partial B$, and $C=C_p$ for $0<p \leq 1$, 
\begin{equation}\label{goal}
V_{C_p,B}(z,w)=\limsup_{|\alpha|\to \infty} \frac{1}{{\rm deg}_{C_{p}}(p_{\alpha})}\log|p_{\alpha}(z,w)|, \quad (z,w)\not\in B.
\end{equation}
We look at points on the diagonal $w=z$ for $|z|$ large; indeed, we may consider any points $(z,w)$ with $|z|=|w|$ large. For $C_1=\Sigma$ we have $V_{C_1,B}(z,w)=\frac{1}{2}\log^+ (|z|^2+|w|^2)$ so for $|z|\geq 1/\sqrt 2$,
            $$V_{C_1,B}(z,z)=\frac{1}{2}\log (|z|^2+|z|^2)=\frac{1}{2}\log2 +\log |z|.$$
Since $\|z^aw^b\|^2_{L^2(\mu)}={a!b!}/{(a+b+1)!}$, at points $(z,z)$ with $|z|>1$ for $C_1$ we need the orthonormal monomials $\{z^aw^b/\|z^aw^b\|_{L^2(\mu)}\}$ with $a,b$ near equal to achieve this value $\frac{1}{2}\log2 +\log |z|$ using (\ref{goal}). Precisely, for $n$ large we need 
            $$\frac{1}{n}\log \|z^aw^b\|^{-1}_{L^2(\mu)}\to \frac{1}{2}\log2.$$
            For $a=b=n/2$ (we assume $n$ even for simplicity), using Stirling's formula,
            $$\|z^{n/2}w^{n/2}\|^{-1}_{L^2(\mu)}= {\sqrt {(n+1)!}\over (n/2)!}\asymp {\sqrt {(n/e)^n}\over (n/2e)^{n/2}}=2^{n/2}$$
            so that 
            $$\frac{1}{n}\log \|z^{n/2}w^{n/2}\|^{-1}_{L^2(\mu)}\asymp \frac{1}{n}\log {2^{n/2}}=\frac{1}{2}\log2$$
            as desired. 
            
If $0<p<1$, the only monomials $z^aw^b\in \Poly(nC_p)$ with $a+b$ ``near'' $n$ are ``near'' $z^n$ and $w^n$ (i.e., corresponding to integer lattice points near the coordinate axes); while those with $a,b$ near equal have $a+b$ ``well away'' from $n$ by concavity of the curve $x^p+y^p =n$. Fix $0< p<1$ and fix any $0<\lambda<1$.
For any monomial $z^aw^b$ with $(a,b)\in n(1-\lambda)C_1$ we have $a+b\leq n(1-\lambda)$ so that at a point $(z,z)$ the function $\frac{1}{n}\log |z^aw^b|$ takes the value
$$\frac1n\log|z|^{a+b}\leq\frac{1}{n}\log |z|^{n(1-\lambda)}=(1-\lambda)\log |z|.$$
Since 
$(1-\lambda)<1$, for points $(z,z)$ with $|z|$ sufficiently large, these monomials cannot approach $\log |z| + C$ regardless of $n$. 
            
For $\lambda$ sufficiently close to 1, any remaining monomials $z^aw^b\in \Poly(nC_p)$ with $(a,b)\in nC_p\setminus n(1-\lambda)C_1$ must have exponents close to $(n,0)$ or $(0,n)$. In the sequel we only consider the first case, the second case being identical. Then, there exists some $0<\lambda_{0}<1/4$, say, such that
\begin{equation}\label{three} 
a>n(1-\lambda_{0}) \ \hbox{and} \ b<n\lambda_{0} \ \hbox{with} \ a+b \leq n.
\end{equation}
            At a point $(z,z)$ with $|z|\geq 1$ the function $\frac{1}{n}\log |z^aw^b|$ has an upper bound of 
$$\frac{1}{n}\log |z^n|=\log |z|$$
so to complete our argument it suffices to show that 
$$\limsup_{n\to \infty} \frac1n\log\|z^{a}w^{b}\|^{-1}_{L^2(\mu)} < \frac{1}{2}\log2$$
where the $\limsup$ is taken over monomials $z^aw^b$ satisfying (\ref{three}). 
Estimating 
$$\|z^{a}w^{b}\|^{-2}_{L^2(\mu)} \simeq {(a+b)!\over a!b!}
\leq\frac{(n\lambda_{0}+1)\cdots n}{a!}
\leq {n!\over (n(1-\lambda_{0}))!(n\lambda_{0})!},$$
and using Stirling's formula,
$${n!\over (n(1-\lambda_{0}))!(n\lambda_{0})!}\asymp {(n/e)^{n}\over (n(1-\lambda_{0})/e)^{n(1-\lambda_{0})}(n\lambda_{0}/e)^{n\lambda_{0}}}
=(1-\lambda_{0})^{-n(1-\lambda_{0})}\lambda_{0}^{-n\lambda_{0}}.$$
Setting $L_n$ for this last expression, we have
$$\frac{1}{n} \log L_n 
=-\left((1-\lambda_{0})\log (1-\lambda_{0})+\lambda_{0} \log \lambda_{0}\right).$$
 We want to show that, for $\lambda_{0}<1/4$, this last quantity is smaller than $\log 2$.
The function $f(x):=-[(1-x)\log (1-x) +x\log x]$ is increasing on $(0,1/2)$, decreasing on $(1/2,1)$, and has a maximum at $x=1/2$ with $f(1/2) = \log 2$; this gives the result.  
\end{proof}

\section{Non-convex Bernstein-Walsh}

We continue to let $C$ be the closure of an open, connected set satisfying (\ref{stdhyp}). Given $K\subset \C^d$ compact, as in the convex setting for $f\in C(K)$ we define $$d_{n}^{C}(f,K)  =\inf_{p\in\Poly(nC)}\|f-p\|_{K}$$ where as before
\begin{align*}
\Poly(nC)=\{p(z) & =\sum_{J\in nC \cap \N^d}c_{J}z^{J}=\sum_{J\in nC \cap \N^d}c_{J}z_1^{j_1}\ldots z_d^{j_d},~c_{J}\in\C\},\quad n=1,2,\ldots\end{align*}
Note that the dimension of $\Poly(nC)$ is proportional to $\hbox{vol}(C) \cdot n^d$ where $\hbox{vol}(C)$ is the $d-$dimensional volume of $C$. In this section, we consider generalizations of Theorem \ref{BW}.

\begin{proposition} \label{oneway} Let $K\subset \C^d$ be compact with $V_{C,K}$ continuous and satisfying (\ref{forPhin}) locally uniformly in $\C^d$. Suppose for some $R>1$ we have $f\in C(K)$ which satisfies 
$$\limsup_{n\to \infty} [d_n^{C}(f,K)]^{1/n} \leq 1/R.$$
Then $f$ extends holomorphically to the open set
$$\Omega_{R,C}:=\{z \in \C^d: V_{C,K}(z) < \log R\}.$$
\end{proposition}

\begin{proof} Under the assumption that 
\begin{equation}\label{locunif2}\lim_{n\to \infty} \frac{1}{n}\log \Phi_n(z) =V_{C,K}(z) \ \hbox{locally uniformly in} \ \C^d,\end{equation}
we have an {\it asymptotic Bernstein-Walsh inequality}: for any $E\subset \C^d$ compact and $\epsilon >0$, we have $n_0=n_0(\epsilon,K)$ so that
$$|p_n(z)|\leq \|p_n\|_Ke^{n(V_{C,K}(z)+\epsilon)}, \ z\in E$$
for any $p\in \Poly(nC)$ and $n>n_0$. This follows since 
$$\frac{1}{n}\log\frac{|p_n(z)|}{\|p_n\|_K}\leq \frac{1}{n} \log \Phi_n(z) \leq V_{C,K}(z)+\epsilon$$
for $z\in E$ and $n>n_0$ by (\ref{locunif2}). We use this to show that if $p_n\in Poly(nC)$ satisfies 
$d_n^{C}(f,K)=\|f-p_n\|_K$, then the series 
$p_0 + \sum_1^{\infty} (p_n-p_{n-1})$ converges
uniformly on compact subsets of $\Omega_{R,C}$ to a holomorphic
function $F$ which agrees with $f$ on
$K$.  To this end, choose $R'$ with 
$1< R' < R$; by hypothesis the polynomials
$p_n$ satisfy
\begin{equation}\label{f=F}
  \|f-p_n\|_K \leq {M \over  {R'} ^n}, \qquad n=0,1,2,...,
\end{equation}
for some $M > 0$. Let $\rho$ satisfy $1 < \rho < R' < R$. Fix $\epsilon >0$ sufficiently small so that $1 < \rho < \rho e^{\epsilon}< R'$, and apply the definition of $V_{C,K}$ and $\Omega_{\rho,C}$ with the asymptotic Bernstein-Walsh estimate on $E=\bar \Omega_{\rho,C}$ and the polynomial $p_n - p_{n-1}\in Poly(nC)$ to obtain 
$$
  \sup _{\bar \Omega_{\rho,C}} |p_n(z) - p_{n-1} (z)|
  \leq 
  \rho^n e^{n\epsilon}\|p_n-p_{n-1}\|_K$$
$$  \leq
  \rho^n e^{n\epsilon}( \|p_n -  f\|_K + \|f - p_{n-1}\|_K ) 
  \leq
  \rho^n  e^{n\epsilon}{M ( 1 + R' ) \over {R'} ^n}.
$$
Since $\rho$ and $R'$ were arbitrary
numbers satisfying $1 < \rho < R' < R$, we
conclude that $ p_0 + \sum_1^{\infty} (p_n-p_{n-1})$ 
converges locally uniformly on $\Omega_{R,C}$ to a holomorphic function $F$. 
From (\ref{f=F}), $F = f$ on $K$.
\end{proof}

The direct converse of Proposition \ref{oneway} is false, in general; simple examples can be constructed using product sets (use Propositions \ref{easy} and \ref{threeeight} (below)). Our goal is to determine what one can say in the opposite direction. If $F$ is holomorphic in a neighborhood of $K\subset \C^d$ compact, then $\limsup_{n\to \infty} d_n^{\Sigma}(F,K)^{1/n}<1$ and from (\ref{stdhyp}) we then have $\limsup_{n\to \infty} d_n^{C}(F,K)^{1/n}<1$. We will use the following lemma, which says that the asymptotic behavior of the rates of polynomial aproximation of $F\in C(K)$ by $\Poly(nC)$ in the sup norm on $K$ and in the $L^{2}$ norm with respect to a Bernstein-Markov measure (\ref{wtdbm}) on $K$ are the same, in this $n$-th root sense. 

\begin{lemma}\label{L2-rate}

Let $K\subset\C^{d}$ be compact and nonpluripolar. Let $\mu$ be a probability measure on $K$ which satisfies the Bernstein-Markov property (\ref{wtdbm}). Let $F\in C(K)$. Let $C$ be a compact, connected subset of $\R_{+}^{d}$ with nonempty interior, such that $a\Sigma\subset C$ for some $a>0$. Assume that
$$
\limsup_{n\to \infty}(d_{n}^{C}(F,K))^{1/n}=:\rho_{\infty}<1.
$$
If $\{p_{n}\}$ is a sequence of best $L^{2}_{\mu}$ approximants to $F$ with $p_n\in \Poly(nC)$ then
\begin{equation}\label{same-rat}
\limsup_{n \to \infty}\|F-p_{n}\|_{\mu}^{1/n}=\limsup_{n \to \infty}\|F-p_{n}\|_{K}^{1/n}=\rho_{\infty}.
\end{equation}
\end{lemma}
\begin{proof}
Let $r$ such that $\rho_{\infty}<r<1$. For $n$ large enough, there exists $q_{n}\in\Poly(nC)$ such that
$$\|F-q_{n}\|_{K}\leq r^{n}.$$
Hence, for the sequence $p_{n}$, we have
$$
\|F-p_{n}\|_{\mu}\leq\|F-q_{n}\|_{\mu}\leq\|F-q_{n}\|_{K}\leq r^{n},
$$
and, in particular, $p_{n}$ converges to $F$ in $L^{2}_{\mu}$. Moreover, for $k$ large,
$$
\|p_{k+1}-p_{k}\|_{\mu}\leq\|F-p_{k+1}\|_{\mu}+\|F-p_{k}\|_{\mu}\leq 2r^{k}.
$$
By the Bernstein-Markov property (\ref{wtdbm}), for a given $\eps>0$ such that $\tilde r=r(1+\eps)<1$, we have, for $k$ large,
$$
\|p_{k+1}-p_{k}\|_{K}\leq (1+\eps)^{k}\|p_{k+1}-p_{k}\|_{\mu}\leq 2\tilde r^{k},
$$
and thus $p_{n}=\sum_{k=-1}^{n-1}(p_{k+1}-p_{k})$ converges uniformly to $F$ on $K$. Moreover,
$$
\|F-p_{n}\|_{K}=\|\sum_{k=n}^{\infty}(p_{k+1}-p_{k})\|_{K}\leq 2
\frac{\tilde r^{n}}{1-\tilde r}.
$$
Letting $r$ tend to $\rho_{\infty}$, $\eps$ tend to 0, and taking $n$-th roots, proves the second equality in (\ref{same-rat}). 
For the first equality, set
$$
\rho_{\mu}:=\limsup_{n\to \infty}\|F-p_{n}\|_{\mu}^{1/n}\leq\rho_{\infty}.
$$
We have, for $\rho_{\mu}<r<1$, $\tilde r=r(1+\eps)$, and $n$ large,
$$
\|F-p_{n}\|_{K}\leq\sum_{k=n}^{\infty}\|p_{k+1}-p_{k}\|_{K}\leq
\sum_{k=n}^{\infty}(1+\eps)^{k}\|p_{k+1}-p_{k}\|_{\mu}
\leq2\sum_{k=n}^{\infty}\tilde r^{k}=\frac{2\tilde r^{n}}{1-\tilde r}.
$$
Letting $r$ tend to $\rho_{\mu}$, $\eps$ tend to 0, and taking $n$-th roots finishes the proof. 
\end{proof}

For simplicity, in the rest of this section we work in $\C^2$. The ``standard'' version of Theorem \ref{BW} is the case where $C$ is the simplex $\Sigma$ defined in (\ref{simpl}) with $d=2$. As a first attempt at a converse to Proposition \ref{oneway}, we let $C_{p}$ be the domain in $\R_+^{2}$ defined by 
$$
C_{p}=\{(x,y)\in\R^{2},~x,y\geq0,~x^{p}+y^{p}\leq1\},\qquad0<p\leq1,
$$
i.e., the $d=2$ case of (\ref{ceepee}). For $0<p<1$, $C_{p}$ is a non-convex body satisfying (\ref{stdhyp}), and for $p=1$, $C_{1}=\Sigma$. For $0<\alpha<1$, let $T_{\alpha}$ be the triangle with vertices $(0,0),(0,\alpha),(\beta,0)$ where $\beta>0$ is such that the side from $(0,\alpha)$ to $(\beta,0)$ is tangent to the curve $x^{p}+y^{p}=1$. Note $\beta$ is a function of $\alpha$.
Let $A_{p}$ and $A_{\alpha}$ denote the square roots of the areas of $C_{p}$ and $T_{\alpha}$, i.e., 
\begin{equation}\label{def-area}
A_{p}^{2}=\frac{\Gamma(1/p)\Gamma(1+1/p)}{p\Gamma(1+2/p)},\qquad A_{\alpha}^{2}=\frac{\alpha\beta}{2},
\end{equation}
where $\Gamma(x)$ denotes the Gamma function.

Our aim is to compare the rates of approximation with respect to a non-convex $C_{p}$ ($0<p<1$), and with respect to the family of inscribed convex triangles $T_{\alpha}$, $0<\alpha<1$. Of course, since $T_{\alpha}\subset C_{p}$, for a compact set $K\subset\C^{2}$ and a function $F$ on $K$, we have
\begin{equation}\label{obv-ineq}
d_{n}^{C_{p}}(F,K)\leq \inf_{0<\alpha<1}d_{n}^{T_{\alpha}}(F,K).
\end{equation}
The comparison is less clear, and more interesting, if we take into account that the number of monomials in $nC_{p}$ is larger than the number of monomials in $nT_{\alpha}$, and thus normalize accordingly, by comparing $d_{n}^{C_{p}}(F,K)^{1/A_{p}}$ and $ d_{n}^{T_{\alpha}}(F,K)^{1/A_{\alpha}}$; i.e., 
\begin{equation}\label{2rates}
d_{n/A_{p}}^{C_{p}}(F,K)\quad\text{and}\quad d_{n/A_{\alpha}}^{T_{\alpha}}(F,K)
\end{equation}
where by abuse of notation, we write $d_{n/A}^C(F,K)$ to denote $d_{\lfloor n/A \rfloor}^C(F,K)$. In the sequel, we estimate these two quantities explicitly in two extreme cases, namely when the function $F$ is of the form
$$F(z,w)=f(z)+g(w)\quad\text{and when}\quad F(z,w)=f(zw).
$$ 
Let us start with the first case, and consider a subset $K=A\times B\subset\C^{2}$ where $A$ and $B$ are regular compact subsets of $\C$. We will denote by
$\mu:=\mu_{A}\otimes\mu_{B}$ the measure on $K$ arising from Bernstein-Markov measures $\mu_{A},\mu_{B}$ on $A,B$.
Let $f$ and $g$ be holomorphic functions in neighborhoods of $A$ and $B$. 
We denote by $\rho_A:=\rho_{A}(f)$ and $\rho_B:=\rho_{B}(g)$ the asymptotic rate (in the $n$-th root sense) of best uniform polynomial approximation to $f$ on $A$ and to $g$ on $B$. Then $\max(\rho_A,\rho_B) <1$.  Define the function of two variables
$$F(z,w)=f(z)+g(w).
$$ 
We begin with a lemma which is applicable to the sets $C_p$ when $0\leq p\leq 1$.
\begin{lemma}\label{L2-sep}
Let $C$ be a subset of $\R_{+}^{2}$ and assume that
$$
C\subset[0,1]^{2},\quad C\cap(\R_{+}\times\{0\})=[0,1]\times\{0\},\quad C\cap(\{0\}\times\R_{+})=\{0\}\times[0,1].
$$
Let 
$P_{n}(z,w)$ be the best $L^{2}_{\mu}$ approximant to $F(z,w)$ in $\Poly(nC)$. Then
$$
P_{n}(z,w)=t_{n}^{f}(z)+t_{n}^{g}(w),
$$
where $t_{n}^{f}$ and $t_{n}^{g}$ are the best $L^{2}_{\mu_{A}}$ and best $L^{2}_{\mu_{B}}$ approximants to $f$ and $g$ in $P_{n}(\C)$, the space of polynomials in one variable of degree less than or equal to $n$.
\end{lemma}
\begin{proof}
Assume that
\begin{align*}
p_{0}(z), p_{1}(z),\ldots,p_{n}(z),\qquad\deg p_{k}=k,\quad k=0,\ldots,n,
\\[5pt]
q_{0}(z), q_{1}(z),\ldots,q_{n}(z),\qquad\deg q_{k}=k,\quad q=0,\ldots,n,
\end{align*}
 are orthonormal bases in $L^{2}_{\mu_{A}}$ and $L^{2}_{\mu_{B}}$. Then, the family of polynomials $p_{k}(z)q_{l}(w)$, $(k,l)\in nC$, is an orthonormal basis of $\Poly(nC)$. Moreover
\begin{align*}
\int f(z)\bar{p_{k}(z)q_{l}(w)}d\mu_{A}(z)d\mu_{B}(w) & =\delta_{l,0}\int f(z)\bar{p_{k}(z)}d\mu_{A}(z),\quad l\geq0,
\end{align*}
and similarly 
\begin{align*}
\int g(w)\bar{p_{l}(z)q_{k}(w)}d\mu_{A}(z)d\mu_{B}(w) & =\delta_{l,0}\int g(w)\bar{q_{k}(w)}d\mu_{B}(w),\quad l\geq0.
\end{align*}
The statement of the lemma follows.
\end{proof}

\begin{remark}\label{alphabeta} If for some $\alpha, \beta \leq 1$, $C\subset[0,1]^{2}$ satisfies
$$C\cap(\R_{+}\times\{0\})=[0,\alpha]\times\{0\},\quad C\cap(\{0\}\times\R_{+})=\{0\}\times[0,\beta], $$
a similar proof shows that $P_{n}(z,w)=t_{\lfloor \alpha n \rfloor}^{f}(z)+t_{\lfloor \beta n \rfloor}^{g}(w)$. 
\end{remark}

We now compute the asymptotic rates of approximation with respect to the sets $C_{p}$ and the family $T_{\alpha}$, $0<\alpha<1$.
\begin{proposition}\label{Rate-C}
We have, for $0<p\leq1$,
\begin{align}
\limsup_{n}d_{n}^{C_{p}}(F,K)^{1/n} & =\max(\rho_{A},\rho_{B}).
\label{ineg-C}
\end{align}
\end{proposition}
\begin{proof}
In view of Lemma \ref{L2-rate}, it is equivalent to estimate the rate of best $L^{2}$ approximation to $F$. Let $P_{n}\in\Poly(nC_{p})$ be the best $L^{2}$ approximants to $F$ with respect to the measure $\mu$.
From Lemma \ref{L2-sep}, we get
$$
\|F-P_{n}\|_{\mu}^{2}=\|f(z)-t_{n}^{f}(z)+g(w)-t_{n}^{g}(w)\|^{2}_{\mu}=\|f-t_{n}^{f}\|^{2}_{\mu_{A}}+\|g-t_{n}^{g}\|^{2}_{\mu_{B}},
$$
where we use the fact that
$$
\int(f-t_{n}^{f})(z)d\mu_{A}(z)=0,\quad\int(g-t_{n}^{g})(w)d\mu_{B}(w)=0.
$$
Next, making use of the one-variable version of Lemma \ref{L2-rate}, we have
$$
\limsup_{n\to \infty}\|f-t^{f}_{n}\|_{\mu_{A}}^{1/n}=\rho_{A},\quad
\limsup_{n\to \infty}\|g-t^{g}_{n}\|_{\mu_{B}}^{1/n}=\rho_{B}.
$$
The proposition follows.
\end{proof}
\begin{proposition}\label{Rate-T}
We have, for $0<p\leq1$,
$$
\limsup_{n}d_{n}^{T_{\alpha}}(F,K)^{1/n} =\max(\rho_{A}^{\alpha},\rho_{B}^{\beta}).
$$
\end{proposition}
\begin{proof}
From Remark \ref{alphabeta}, we now get
$$
\|F-P_{n}\|_{\mu}^{2}=\|f(z)-t_{\lfloor \alpha n \rfloor}^{f}(z)+g(w)-t_{\lfloor \beta n \rfloor}^{g}(w)\|^{2}_{\mu}=\|f-t_{\lfloor \alpha n \rfloor}^{f}\|^{2}_{\mu_{A}}+\|g-t_{\lfloor \beta n \rfloor}^{g}\|^{2}_{\mu_{B}}.
$$
Since
$$
\limsup_{n\infty}\|f-t^{f}_{\lfloor \alpha n \rfloor}\|_{\mu_{A}}^{1/n}=\rho_{A}^{\alpha},\quad
\limsup_{n\infty}\|g-t^{g}_{\lfloor \beta n \rfloor}\|_{\mu_{B}}^{1/n}=\rho_{B}^{\beta},
$$
the proposition follows.
\end{proof}
Now recall that comparing the limits of the $n$-th roots of the two rates in (\ref{2rates}) is equivalent to comparing
$$
\limsup_{n}d_{n}^{C_{p}}(F,K)^{1/(A_{p}n)}=\max(\rho_{A},\rho_{B})^{1/A_{p}}
$$
and
$$ \limsup_{n}d_{n}^{T_{\alpha}}(F,K)^{1/(A_{\alpha}n)}=
(\max(\rho_{A}^{\alpha},\rho_{B}^{\beta}))^{1/A_{\alpha}}.
$$
\begin{theorem}\label{main-thm}
When $\alpha$ and $\beta$ are close to each other, that is close  to $(1/2)^{1/p-1}$, and the triangle $T_{\alpha}$ is close to an isosceles triangle, one has
\begin{equation}\label{main1}
\limsup_{n}(d_{n/A_{\alpha}}^{T_{\alpha}}(F,K))^{1/n}\leq
\limsup_{n}(d_{n/A_{p}}^{C_{p}}(F,K))^{1/n},
\end{equation}
while, when $\alpha$ is close to 0 and $\beta$ close to $1$ (or the reverse), and the triangle $T_{\alpha}$ becomes very small, one has the opposite inequality
$$\limsup_{n}(d_{n/A_{p}}^{C_{p}}(F,K))^{1/n}\leq\limsup_{n}(d_{n/A_{\alpha}}^{T_{\alpha}}
(F,K))^{1/n}.$$
\end{theorem}
\begin{proof}
By symmetry, we may assume without loss of generality that $\rho_{B}\leq\rho_{A}$. 

Assume first that $\alpha=\beta$. Then, we have to show that $\rho_{A}^{\sqrt{2}}\leq\rho_{A}^{1/A_{p}}$, that is
$1/A_{p}\leq\sqrt{2}$. In view of (\ref{def-area}), this is equivalent to
$$
p\Gamma(1+2/p)\leq2\Gamma(1/p)\Gamma(1+1/p)\iff\Gamma(1+2/p)\leq2\Gamma(1+1/p)^{2},
$$
which is easily seen to be true by computing derivatives. Moreover, equality holds only if $p=1$.
Thus, by continuity, if $p\neq1$, inequality (\ref{main1}) still holds when $\alpha$ and $\beta$ are close to each other.

Now, consider the case when $\alpha$ is close to 0 and $\beta$ is close to 1. Then $\rho_{A}^{\alpha}$ is close to 1 while $\rho_{B}^{\beta}$ is close to $\rho_{B}$, so $\max(\rho_{A}^{\alpha},\rho_{B}^{\beta})=\rho_{A}^{\alpha}$. Moreover $1/A_{\alpha}\simeq\sqrt{2/\alpha}$, so the limit of the rates corresponding to $T_{\alpha}$ is close to 1, while the limit of the rates corresponding to $C_{p}$ remains a number less than 1.  
\end{proof}
We now consider the case of a function $F(z,w)$ of the form
$$
F(z,w)=f(zw),
$$
where we assume that the largest disk centered at the origin in $\C$ contained in the domain of analyticity of $f$ is the disk $D_{R}$ of radius $R>1$. Also let
$K=\D\times \D=\{(z,w),~|z|\leq 1,~|w|\leq 1\}$ be the unit polydisk in $\C^2$.
\begin{proposition} \label{threeeight}
We have, for $0<p\leq1$, and with $r=1/R$,
\begin{align}
\limsup_{n}d_{n}^{C_{p}}(F,K)^{1/n} =r^{(1/2)^{1/p}}.
\label{ineg-C}
\end{align}
\end{proposition}
\begin{proof}
Let $\mu_{K}=\mu\otimes\mu$, $\mu=d\theta/2\pi$, be the normalized measure supported on $\T\times\T=\{(z,w),~|z|= 1,~|w|= 1\}$. In view of Lemma \ref{L2-rate}, it is sufficent to consider a sequence of best $L^{2}(\mu_K)$ approximants to $F$. Since the family $\{z^{j}w^{k}\}$, $(j,k)\in\N^{2}$, is orthogonal with respect to $\mu_{K}$, we have
$$
\limsup_{n}d_{n}^{C_{p}}(F,K)^{1/n}=\limsup_{n}\|f-p_{na_{p}}\|_{\mu}^{1/n},
$$
where $a_{p}=(1/2)^{1/p}$ is such that point $(a_{p},a_{p})$ is the intersection of the curve $x^{p}+y^{p}=1$ and the line $x=y$, and $p_{na_{p}}$ denotes the best $L^{2}(\mu)$ polynomial approximant to $f$ of degree at most $na_{p}$. Moreover, by using the one variable versions of Lemma \ref{L2-rate} and Theorem \ref{BW}, we get
$$
\limsup_{n}\|f-p_{na_{p}}\|_{\mu}^{1/n}=\limsup_{n}d_{na_{p}}(f,\D)^{1/n}=r^{a_{p}},
$$
which proves (\ref{ineg-C}).
\end{proof}
\begin{proposition}\label{Rate-T2}
We have, for $0<p\leq1$,
\begin{align}\label{ineg-T2}
\limsup_{n}d_{n}^{T_{\alpha}}(F,K)^{1/n} & =r^{\alpha\beta/(\alpha+\beta)}.
\end{align}
\end{proposition}
\begin{proof}
The proof is identical to the proof of {Proposition \ref{threeeight}}. The only change is that we now need to consider the point which is at the intersection of the line $\beta y=\alpha(\beta-x)$, the side of the triangle tangent to $C_{p}$, and the line $x=y$. This point has both coordinates equal to $\alpha\beta/(\alpha+\beta)$, which implies the result.
\end{proof}
From the two previous propositions, we may make more precise inequality (\ref{obv-ineq}), and also compare the normalized rates of approximation.
\begin{theorem}\label{main-thm2}
The following holds true, for $0<p\leq1$,
\begin{equation}\label{main2}
\limsup_{n}(d_{n}^{C_{p}}(F,K))^{1/n}=\inf_{\alpha}\limsup_{n}(d_{n}^{T_{\alpha}}
(F,K))^{1/n}.
\end{equation}
The inf on the right-hand side is attained when
$\alpha=\alpha_{p}=(1/2)^{1/p-1}$, which corresponds to the isosceles triangle $T_{\alpha}$ such that $\alpha=\beta$. For the rates normalized by the areas of $C_{p}$ and $T_{\alpha}$, we have the following. If $\alpha$ and $\beta$ are close to each other, 
$$
\limsup_{n}(d_{n/A_{\alpha}}^{T_{\alpha}}(F,K))^{1/n}\leq\limsup_{n}(d_{n/A_{p}}^{C_{p}}(F,K))^{1/n},
$$
while if one of $\alpha$ or $\beta$ is close to 0, we have the opposite inequality
$$
\limsup_{n}(d_{n/A_{p}}^{C_{p}}(F,K))^{1/n}\leq\limsup_{n}(d_{n/A_{\alpha}}^{T_{\alpha}}(F,K))^{1/n}.
$$
\end{theorem}
\begin{proof}
In view of (\ref{ineg-T2}), the inf on the right of (\ref{main2}) is attained when $\alpha=\beta$. We obtain the value $r^{a_{p}}$ which equals the $\limsup$ on the left. For the normalized rate, and $\alpha$ and $\beta$ close to each other, the asserted inequality is just a consequence of the fact that
$1/A_{p}<1/A_{\alpha}$. When $\alpha$ or $\beta$ is close to 0, one may argue as in the proof of Theorem \ref{main-thm}.
\end{proof}

\begin{example}\label{silly} Related to the previous class of functions $F(z,w)=f(zw)$, simple examples show that in the limiting case $p=0$, the corresponding polynomial classes are too sparse to uniformly approximate even simple bivariate polynomials.  We offer a simple geometric argument to show $f(x,y)=xy$ is not uniformly approximable on $[0,1]\times [0,1]$ by a sum of univariate polynomials in $x$ and in $y$. Indeed, suppose, given $\epsilon >0$, one could find $p(x)$ of degree $n$, say, and $q(y)$ of degree $n$, say, with 
                  $$|p( x)+q(y)-xy| <\epsilon \ \hbox{for} \ 0\leq x,y \leq 1.$$
                  Then for each fixed $y_0\in [0,1]$, 
                  $$|p(x)-[y_0x -q(y_0)]|<\epsilon \   \hbox{for} \ 0\leq x \leq 1.$$
                  This says that the function $p(x)$ simultaneously uniformly approximates the whole family of linear functions $l_{y_0}(x):=y_0x -q(y_0)$ for $0\leq y_0 \leq 1$ on the interval $[0,1]$ (in the $x-$variable) which is impossible (note the slopes of the $l_{y_0}$ vary from $0$ to $1$).
                  \end{example}

\section{Non-convex random polynomials} Let $C$ be the closure of an open, connected set satisfying (\ref{stdhyp}) and which contains $C_{0}$ in (\ref{ceez}). We let $K$ be a nonpluripolar compact set in 
$\C^d$ satisfying (\ref{forPhin}). We assume, moreover, that  
\begin{equation}\label{vck} V_{C,K} \ \hbox{is continuous; i.e.}, \ V_{C,K}=V_{C,K}^*.
\end{equation}
Let $\tau$ be a probability measure on $K$ such that $(K,\tau)$ satisfies (\ref{wtdbm}). Letting $\{p_j\}$ be an orthonormal basis in $L^2(\tau)$ for $\Poly(nC)$ constructed via Gram-Schmidt applied to an ordered monomial basis $\{z^{\nu}\}$ of $\Poly(nC)$, as described in the introduction we consider random polynomials of $C-$degree at most $n$ of the form 
$$H_n(z):=\sum_{j=1}^{m_n} a_{j}^{(n)}p_j(z)$$
where the $a_{j}^{(n)}$ are i.i.d. complex random variables with a distribution $\phi$ satisfying (\ref{hyp1}) and (\ref{hyp2}). Here $m_n=$dim$(\Poly(nC)$). This gives a probability measure $\mathcal H_n$ on $\Poly(nC)$ and we form the product probability space
$$\mathcal H:=\otimes_{n=1}^{\infty} (\Poly(nC),\mathcal H_n)$$
of sequences of random polynomials. We identify $\mathcal H$ with $\mathcal C:=\otimes_{n=1}^{\infty}(\C^{m_n},Prob_{m_n})$ where, for $G\subset \C^{m_n}$,
$$Prob_{m_n}(G):=\int_G \phi(z_1) \cdots \phi(z_{m_n})dm_2(z_1) \cdots dm_2(z_{m_n}).$$
In this setting, we recall a result from \cite{blrp}. Here, we write 
$$a^{(m_n)}=(a_{1}^{(n)},...,a_{m_n}^{(n)})\in \C^{m_n}$$ 
and $<\cdot,\cdot>, \ \| \cdot \|$ denote the standard Hermitian inner product and associated norm on $\C^{m_n}$.

\begin{corollary} \label{useful} Let $\{w^{(m_n)}=(w_{1}^{(n)},...,w_{m_n}^{(n)})\}$ be a sequence of vectors $w^{(m_n)}\in \C^{m_n}$. For $\phi$ satisfying (\ref{hyp1}) and (\ref{hyp2}), with probability one in $\mathcal C$, if $\{m_n\}$ is a sequence of positive integers with $m_n=\OO(n^M)$ for some $M$, then 
\begin{equation}\label{limsupeqn} 
\forall \{w^{(m_n)}\},\quad \limsup_{n\to \infty} \frac{1}{n}\log |<a^{(m_n)},w^{(m_n)}>|\leq \limsup_{n\to \infty} \frac{1}{n}\log \|w^{(m_n)}\|.
\end{equation}
Moreover, for each $\{w^{(m_n)}\}$, 
\begin{equation}\label{liminfeqn}\liminf_{n\to \infty} \frac{1}{n}\log |<a^{(m_n)},w^{(m_n)}>|\geq \liminf_{n\to \infty} \frac{1}{n}\log \|w^{(m_n)}\|\end{equation}
with probability one in $\mathcal C$; i.e., for each $\{w^{(m_n)}\}$, the set 
$$\{  \{a^{(m_n)}:=(a_{1}^{(n)},...,a_{m_n}^{(n)})\}_{n=1,2,...}\in \mathcal C:(\ref{liminfeqn}) \ \hbox{holds}\}$$
depends on $\{w^{(m_n)}\}$ but is always of probability one.

\end{corollary}

 Using Proposition \ref{touse} and Corollary \ref{useful}, we can follow the proof of Theorem 4.1 of \cite{blrp}.

\begin{theorem} \label{pointa} Let $K$ satisfy (\ref{forPhin}) and (\ref{vck}) and let $a_{j}^{(n)}$ be i.i.d. complex random variables with a distribution $\phi$ satisfying (\ref{hyp1}) and (\ref{hyp2}). Then almost surely in $\mathcal H$ we have
$$\bigl(\limsup_{n\to \infty}\frac{1}{n}\log |H_n(z)|\bigr)^*=V_{C,K}(z),\quad z\in\C^{d}.$$

\end{theorem}

\begin{proof} Using the first part of Corollary \ref{useful}, (\ref{limsupeqn}), with
$$w^{(n)}=p^{(n)}(z):= (p_1(z),...,p_{m_n}(z))\in \C^{m_n},$$ almost surely in $\mathcal H$  
\begin{equation}\label{newlimsup} \limsup_{n\to \infty}\frac{1}{n}\log |H_n(z)|\leq V_{C,K}(z),
\quad z\in\C^{d}
\end{equation}
from Proposition \ref{touse}. Fix a countable dense subset $\{z_t\}_{t\in S}$ of $\C^d$. Using the second part of Corollary \ref{useful}, (\ref{liminfeqn}), for each $z_t$, almost surely in $\mathcal H$ we have
\begin{equation}\label{newliminf}\liminf_{n\to \infty}\frac{1}{n}\log |H_n(z_t)|\geq V_{C,K}(z_t).\end{equation}
A countable intersection of sets of probability one is a set of probability one; thus (\ref{newliminf}) holds almost surely in $\mathcal H$ for each $z_t, \ t\in S$.

Define
$$H(z):=\bigl(\limsup_{n\to \infty}\frac{1}{n}\log |H_n(z)|\bigr)^*.$$
From (\ref{newlimsup}), since $V_{C,K}=V_{C,K}^*$, almost surely in $\mathcal H$, $H(z)\leq V_{C,K}(z)$ for all $z\in \C^d$. Moreover, from (\ref{unifbd}), almost surely in $\mathcal H$ we have $\{\frac{1}{n}\log |H_n(z)|\}$ is locally bounded above and hence $H$ is plurisubharmonic; indeed, $H\in L(\C^d)$. By (\ref{newliminf}), $H(z_t)\geq V_{C,K}(z_t)$ for all $t\in S$. Now given $z\in \C^d$, let $S'\subset S$ with $\{z_t\}_{t\in S'}$ converging to $z$. Then, since $V_{C,K}$ is continuous at $z$, 
$$V_{C,K}(z)=\lim_{t\in S', \ z_t \to z}V_{C,K}(z_t)\leq \limsup_{t\in S', \ z_t \to z}H(z_t)\leq H(z).$$ 
Thus $H(z)=V_{C,K}(z)$ for all $z\in \C^d$.

\end{proof}

To obtain convergence of linear differential operators applied to $\frac{1}{n}\log |H_n(z)|$, we verify convergence to $V_{C,K}(z)$ in $L_{loc}^1(\C^d)$.

\begin{theorem} \label{l1} Let $K$ satisfy (\ref{forPhin}) and (\ref{vck}) and let $a_{j}^{(n)}$ be i.i.d. complex random variables with a distribution $\phi$ satisfying (\ref{hyp1}) and (\ref{hyp2}). Then almost surely in $\mathcal H$ we have
$$\lim_{n\to \infty}\frac{1}{n}\log |H_n(z)|=V_{C,K}(z)$$
in $L_{loc}^1(\C^d)$ and hence
$$\lim_{n\to \infty}dd^c\bigl(\frac{1}{n}\log |H_n(z)|\bigr)=dd^cV_{C,K}(z)$$
as positive currents, where $dd^c=\frac{i}{\pi}\partial \bar \partial$.
\end{theorem}

As in \cite{blrp}, the proof of Theorem \ref{l1} will follow from Proposition \ref{thmdet} and a modification of the proof of Theorem \ref{pointa}.

\begin{proof}[Proof of Theorem \ref{l1}] From Proposition \ref{thmdet}, we need to show almost surely in $\mathcal H$ that for any subsequence $J$ of positive integers, we have
$$\bigl(\limsup_{n\in J}\frac{1}{n}\log |H_n(z)|\bigr)^*=V_{C,K}(z)$$
for all $z\in \C^d$. Fix any subsequence $J$. Following the proof of Theorem \ref{pointa}, almost surely in $\mathcal H$  
$$\limsup_{n\in J}\frac{1}{n}\log |H_n(z)|\leq \limsup_{n\to \infty}\frac{1}{n}\log |H_n(z)|\leq V_{C,K}(z)$$
for all $z\in \C^d$ from (\ref{newlimsup}) and the fact that $J$ is a subsequence of positive integers. Fix a countable dense subset $\{z_t\}_{t\in S}$ of $\C^d$.  Then for each $z_t$, almost surely in $\mathcal H$ we have
$$\liminf_{n\in J}\frac{1}{n}\log |H_n(z)|\geq \liminf_{n\to \infty}\frac{1}{n}\log |H_n(z_t)|\geq V_{C,K}(z_t)$$
from (\ref{newliminf}) and the fact that $J$ is a subsequence of positive integers. This relation holds almost surely in $\mathcal H$ for each $z_t, \ t\in S$. 

Now define
$$H_J(z):=\bigl(\limsup_{n\in J}\frac{1}{n}\log |H_n(z)|\bigr)^*.$$
Then almost surely in $\mathcal H$, $H_J$ is plurisubharmonic and $H_J(z)\leq V_{C,K}(z)$ for all $z\in \C^d$; and $H_J(z_t)\geq V_{C,K}(z_t)$ for all $t\in S$. Given $z\in \C^d$, let $S'\subset S$ with $\{z_t\}_{t\in S'}$ converging to $z$. Then
$$V_{C,K}(z)=\lim_{t\in S', \ z_t \to z}V_{C,K}(z_t)\leq \limsup_{t\in S', \ z_t \to z} H_J(z_t)\leq H_J(z).$$ 
Thus $H_J(z)=V_{C,K}(z)$ for all $z\in \C^d$.

\end{proof}

We write $ Z_{H_n}:= dd^c\log |H_n|$ and $ \tilde Z_{H_n}:=\frac{1}{n}dd^c\log |H_n|$, the normalized zero current of $H_n$. The expectation $\mathbb{E}(\tilde Z_{H_n})$ of $ \tilde Z_{H_n}$ is a positive current of bidegree $(1,1)$ defined as follows: the action of $\mathbb{E}(\tilde Z_{H_n})$ on a $(d-1,d-1)$ form $\alpha$ with $C_0^{\infty}(\C^d)$ coefficients is given as the average of the action $\bigl(\tilde Z_{H_n},\alpha \bigr)$ of the normalized zero current $\tilde Z_{H_n}$ on $\alpha$:
$$\bigl(\mathbb{E}(\tilde Z_{H_n}),\alpha\bigr):=\int_{\C^{m_n}}\bigl(\tilde Z_{H_n},\alpha \bigr)dProb_{m_n}(a^{(n)})=\int_{\C^{m_n}}\bigl(\frac{1}{n}dd^c\log |H_n|,\alpha \bigr)dProb_{m_n}(a^{(n)}).$$
Using Theorem \ref{l1}, we verify the analogue of Theorem 7.1 of \cite{blrp}.

\begin{theorem} \label{keythma} Let $K$ satisfy (\ref{forPhin}) and (\ref{vck}) and let $a_{j}^{(n)}$ be i.i.d. complex random variables with a distribution $\phi$ satisfying (\ref{hyp1}) and (\ref{hyp2}). Then $\lim_{n\to \infty} \mathbb{E}(\tilde Z_{H_n})=dd^cV_{C,K}$ as positive $(1,1)$ currents.
\end{theorem}

\begin{proof} Theorem \ref{l1} gives
$$\lim_{n\to \infty}dd^c\bigl(\frac{1}{n}\log |H_n|\bigr)=dd^cV_{C,K}$$
as positive currents a.s. in $\mathcal H$. We want to show
$$\lim_{n\to \infty}\bigl(\mathbb{E}(\tilde Z_{H_n}),\alpha\bigr)=\bigl(dd^cV_{C,K},\alpha \bigr)$$
for each $(d-1,d-1)$ form $\alpha$ with $C_0^{\infty}(\C^d)$ coefficients. 
As before, we write $a^{(n)}$ for the $m_n-$tuple $\{a_j^{(n)}\}$ of coefficients of $H_n$. Given $\alpha$, define 
$$f_n=f_n^{(\alpha)}:\C^{m_n}\to \C \ \hbox{as} \ f_n(a^{(n)}):=\bigl(\tilde Z_{H_n},\alpha \bigr).$$
Then $\{f_n\}$ are uniformly bounded by the norm of $\alpha$ on its support and extending $f_n$ to $F_n$ on ${\mathcal H}$ via 
$$F_n(\cdots,a^{(n)},\cdots):=f_n(a^{(n)}),$$ 
the $\{F_n\}$ are uniformly bounded on ${\mathcal H}$. Define
$$\int_{\mathcal H} F_n(\cdots,a^{(n)},\cdots)\otimes_{n=1}^{\infty}dProb_{m_n}(a^{(n)})=\int_{\C^{m_n}}f_n(a^{(n)})dProb_{m_n}(a^{(n)}).$$
We apply dominated convergence to $\{F_n\}$ on ${\mathcal H}$ to conclude.

\end{proof}

\begin{remark} Following \cite{blrp}, using different arguments one can eliminate the need for (\ref{hyp1}) in Theorem \ref{keythma}. Moreover, one can slightly weaken the hypothesis (\ref{hyp2}) as in \cite{Bay}.

\end{remark}

For $2\leq k\leq d$, we consider
the common zeros of $k$ polynomials $H_n^{(1)},...,H_n^{(k)}$  where 
$$H_n^{(l)}(z):=\sum_{j=1}^{m_n} a_{j}^{(n,l)}p_j(z), \ l=1,...,k$$ 
with the $a_{j}^{(n,l)}$ i.i.d. complex random variables with a distribution $\phi$ satisfying (\ref{hyp1}) and (\ref{hyp2}). For $k=2,3,...,d$, we observe that the wedge product
$$Z^k_{{\bf H}_n}:=dd^c\log |H_n^{(1)}|\wedge \cdots \wedge dd^c\log |H_n^{(k)}|$$
is a.s. well-defined as a positive $(k,k)$ current; cf., \cite{Bay} or \cite{blrp}. We write
$\tilde Z^k_{{\bf H}_n}=({1}/{n^k})Z^k_{{\bf H}_n}
$
for the normalized zero current. The expectation $\mathbb{E}(\tilde Z^k_{{\bf H}_n})$ of $\tilde Z^k_{{\bf H}_n}$ is a positive $(k,k)$ current: for $k=2,3,...,d$, the action of $\mathbb{E}(\tilde Z_{{\bf H}_n})$ on a $(d-k,d-k)$ form $\alpha$ with $C_0^{\infty}(\C^d)$ coefficients is given as the average of the action $\bigl(\tilde Z_{{\bf H}_n},\alpha \bigr)$ of the normalized zero current $\tilde Z_{{\bf H}_n}$ on $\alpha$: writing $a^{(n,l)}=(a_1^{(n,l)},...,a_{m_n}^{(n,l)}), \ l=1,...,k$,
$$\bigl(\mathbb{E}(\tilde Z_{{\bf H}_n}),\alpha\bigr):=\int_{(\C^{m_n})^k}\bigl(\tilde Z_{{\bf H}_n},\alpha \bigr)dProb_{m_n}(a^{(n,1)})\cdots dProb_{m_n}(a^{(n,k)})$$
$$=\int_{(\C^{m_n})^k}\bigl(dd^c\log |H_n^{(1)}|\wedge \cdots \wedge dd^c\log |H_n^{(k)}|,\alpha \bigr)dProb_{m_n}(a^{(n,1)})\cdots dProb_{m_n}(a^{(n,k)}).$$
If, in addition, the distribution $\phi$ is smooth (e.g., $\phi(z)=\frac{1}{\sqrt {2\pi}}e^{-\pi |z|^2}$, a standard complex Gaussian), by the independence of $H_n^{(1)},...,H_n^{(k)}$ we have 
$$\mathbb{E}(\tilde Z^k_{{\bf H}_n})=\mathbb{E}(\frac{1}{n^k}dd^c\log |H_n^{(1)}|\wedge \cdots \wedge dd^c\log |H_n^{(k)}|)$$
$$=\mathbb{E}(\tilde Z_{H_n^{(1)}})\wedge \cdots \wedge \mathbb{E}(\tilde Z_{H_n^{(k)}})=[\mathbb{E}(\tilde Z_{H_n^{(1)}})]^k$$
(cf., the argument in Corollary 3.3 of \cite{ba}). Thus from Theorem \ref{keythma} we obtain the asymptotics of these $(k,k)$ currents.

\begin{corollary} \label{cora} Let $K$ satisfy (\ref{forPhin}) and (\ref{vck}) and let $a_{j}^{(n,l)}$ be i.i.d. complex random variables with a smooth distribution $\phi$ satisfying (\ref{hyp1}) and (\ref{hyp2}). Then for $k=2,...,d$,
\begin{equation}\label{exp} \lim_{n\to \infty} \mathbb{E}(\tilde Z^k_{{\bf H}_n})=\lim_{n\to \infty} \mathbb{E}(\frac{1}{n^k}dd^c\log |H_n^{(1)}|\wedge \cdots \wedge dd^c\log |H_n^{(k)}|)=(dd^cV_{C,K})^k.\end{equation}
\end{corollary}

Taking $K=E_1\times \cdots \times E_d$, a product of regular planar compacta $E_j$, we have 
$$V_{C,K}(z_1,...,z_d)=\max_{j=1,...,d}g_{E_j}(z_j).$$
Letting $k=d$ in Corollary \ref{cora}, we have the following result.

\begin{corollary} \label{expzero1} For $K=E_1\times \cdots \times E_d$, a product of regular planar compacta $E_j$, we have 
$$\lim_{n\to \infty} \mathbb{E}(\tilde Z^d_{{\bf H}_n})=(dd^cV_{C,K})^d=\otimes_{j=1}^d \mu_{E_j}$$
where $\mu_{E_j}=\Delta g_{E_j}$.

\end{corollary}

\begin{remark} This holds, e.g., for $C_p$ in (\ref{ceepee}) for all $0\leq p \leq 1$ (for $p=0$, see Remark \ref{cob2} below). For $p=0$ the classes $\Poly(nC_0)$ are very sparse -- they consist of sums $\sum_{j=1}^dp_j(z_j)$ of univariate polynomials $p_j$ of degree at most $n$ while for $p=1$ the classes $\Poly(nC_1)=\Poly(n\Sigma)$ are the ``standard'' polynomials of degree at most $n$. On the other hand, for $p>1$ the set $C_p$ is convex, and from (\ref{tabform}) and (\ref{phic}), if we let $1/p+1/q=1$,
$$V_{C_p,K}(z_1,...,z_d)=\phi_{C_p}(g_{E_1}(z_1),...,g_{E_d}(z_d))=[g_{E_1}(z_1)^{q}+\cdots +g_{E_d}(z_d)^{q}]^{1/q}.$$
The results in \cite{Bay} show that the expected normalized zero measures $\mathbb{E}(\tilde Z^d_{{\bf H}_n})$ for the random polynomial mappings in this setting converge to 
$$(dd^cV_{C_p,K})^d=dd^c\bigl([g_{E_1}(z_1)^{q}+\cdots +g_{E_d}(z_d)^{q}]^{1/q})\bigr)^d$$
which clearly changes with $p$.
\end{remark}

\begin{remark} \label{cob2} As in subsection 2.3, if $\mu$ is a Bernstein-Markov measure on a nonpluripolar compact set $K\subset \C^d$ satisfying (\ref{forPhin}) and (\ref{vck}), the inequality (\ref{step2}) is valid for $C_0,K$ and $\mu$ and all the results of this section are valid. In particular, we can take $K=B:=\{z\in \C^d: |z_1|^2+ \cdots +|z_d|^2\leq 1\}$, the complex Euclidean ball in $\C^d$ and $\mu_B$ normalized surface area measure on $\partial B$, or $K=T:=\{z\in \C^d: |z_1|= \cdots = |z_d| = 1\}$ the unit torus and $\mu_T$ normalized Haar measure on $T$. From Proposition \ref{easy} and (\ref{vcob}), the $C_0-$extremal functions are the same: $V_{C_0,K}(z)=\max [0,\log|z_1|,...,\log |z_d|]$. Thus {\it in both cases the corresponding expected normalized zero measures $\mathbb{E}(\tilde Z^d_{{\bf H}_n})$ converge to $(dd^c V_{C_0,K})^d=\mu_T$.}

\end{remark}

\section{Questions and further directions} The reader will note that many basic issues in the non-convex theory are unresolved. We include a partial list. In 1. and 2. $C$ is the closure of an open, connected set satisfying (\ref{stdhyp}) and $K$ is a compact set in  $\C^d$.

\begin{enumerate}
\item Do we have equality in (\ref{vgevc}), i.e., does 
$$V_{C,K}(z)= \sup\{\frac{1}{{\rm deg}_C(p)}\log |p(z)|:p\in \PP_{d}, \ \|p\|_K\leq 1\}?$$

\item Does the limit in (\ref{forPhin}) 
$$\lim_{n\to \infty} \frac{1}{n}\log \Phi_n(z) =V_{C,K}(z)$$
exist for $z\in \C^d$?

\item For the complex Euclidean ball $B\subset \C^d$, are the $C_p-$extremal functions $V_{C_p,B}$ different for different $p\in (0,1)$? Proposition \ref{propb} simply asserts that $V_{C_p,B}(z) \not = V_{C_{0},B}(z), \ V_{C_{1},B}(z)$ at certain points $z\in \C^d$.

\item For $A, B\subset \C$ and $0< p \leq 1$ we saw that  
$V_{C_p,A\times B}(z,w)=\max[g_A(z),g_B(w)]$. On the other hand, for the triangles $T_{\alpha}$ defined before \ref{def-area}, we have $V_{T_{\alpha},A\times B}(z,w)=\max(\beta g_A(z),\alpha g_B(w))$ (cf., Proposition 2.4 of \cite{BosLev}) so that 
$$V_{C_p,A\times B}(z,w)=\sup_{0<\alpha <1}V_{T_{\alpha},A\times B}(z,w).$$ 
Is the equality 
$$V_{C_p,K}(z,w)=\sup_{0<\alpha <1} V_{T_{\alpha},K}(z,w)$$
true for more general $K\subset \C^2$, e.g., is this true  for the complex Euclidean ball $B\subset \C^2$?


\end{enumerate}


\vspace{1cm}
{\obeylines
\texttt{N. Levenberg, nlevenbe@indiana.edu
Indiana University, Bloomington, IN 47405 USA
\medskip
F. Wielonsky, franck.wielonsky@univ-amu.fr
Laboratoire I2M - UMR CNRS 7373
Universit\'e Aix-Marseille, CMI 39 Rue Joliot Curie
F-13453 Marseille Cedex 20, FRANCE }
}

\end{document}